\documentclass{article}

\usepackage[english]{babel} % English formatting
\usepackage[utf8]{inputenc} % Standard encoding
\usepackage[a4paper,left=2.5cm,right=2.5cm,bottom=3cm]{geometry} % Page formatting
\usepackage{indentfirst} % Indents the first paragraph
\usepackage{amsmath,amssymb,amsfonts,amsthm} % Maths type package
\numberwithin{equation}{section}
\usepackage{bm} % Bold font maths
\usepackage{graphicx} % Advanced graphics package
\usepackage[export]{adjustbox} 
\usepackage{pdflscape} % Make pages landscape
\usepackage{fancyhdr} % Fancy headers
\usepackage[colorlinks=true,linkcolor=cyan,citecolor=viridian]{hyperref} % Link colours
\usepackage{natbib} % Bibliography
\setcitestyle{numbers,square}
\usepackage{flafter} % Reference any 'float'
\usepackage[framemethod=tikz]{mdframed} % Box off stuff
\usepackage{color} % Colour support
	\definecolor{yellow-green}{rgb}{0.6, 0.8, 0.2}
	\definecolor{viridian}{rgb}{0.25, 0.51, 0.43}
\usepackage{wrapfig} % Text flowing around figures
\usepackage{lipsum} % Generates meaningless text
\usepackage{prettyref}
\usepackage[all]{xy}
\usepackage[auth-lg,noblocks]{authblk}
\usepackage{framed}
\usepackage{tikz}
\usepackage{quiver}
\usetikzlibrary{hobby}
\usepackage{thmtools}
\usepackage{thm-restate}

\graphicspath{ {Images/} } % Folder for images 

\usepackage{multicol} % Multi Column Environment
\usepackage{float} % Lets you add images to multi-column environment

\usepackage{nicematrix}
\usetikzlibrary{patterns}
\usetikzlibrary{matrix,decorations.pathreplacing}
\usetikzlibrary{decorations.pathreplacing,calligraphy}
\title{Non-maximal Anosov Representations from Surface Groups to $\mathrm{SO}_0(2,3)$}

\author{Junming Zhang\thanks{Chern Institute of Mathematics, Nankai University, Tianjin 300071, China, \href{mailto:junmingzhang@mail.nankai.edu.cn}{junmingzhang@mail.nankai.edu.cn}}}

\date{\vspace{-2em}}

\newtheorem{theorem}{Theorem}[section]
\newtheorem{corollary}[theorem]{Corollary}
\newtheorem{proposition}[theorem]{Proposition}
\newtheorem{lemma}[theorem]{Lemma}
\newtheorem{definition}[theorem]{Definition}
\newtheorem{remark}[theorem]{Remark}
\newtheorem{example}[theorem]{Example}
\newtheorem{fact}[theorem]{Fact}

\newtheorem{question}[theorem]{Question}
\newrefformat{lemma}{Lemma \ref{#1}}
\newrefformat{coro}{Corollary \ref{#1}}
\newrefformat{thm}{Theorem \ref{#1}}
\newrefformat{prop}{Proposition \ref{#1}}
\newrefformat{rem}{Remark \ref{#1}}
\newrefformat{example}{Example \ref{#1}}
\newrefformat{defn}{Definition \ref{#1}}
\newrefformat{section}{Section \ref{#1}}
\newrefformat{conj}{Conjecture \ref{#1}}
\newrefformat{apdx}{Appendix \ref{#1}}
\newrefformat{fact}{Fact \ref{#1}}
\newrefformat{question}{Question \ref{#1}}
\newcommand{\dd}{{\mathrm d}}
\newcommand{\EE}{{\mathcal{E}}}

\newcommand{\KK}{{\mathcal{K}}}
\newcommand{\LL}{{\mathcal{L}}}

\newcommand{\II}{{\mathcal{I}}}
\newcommand{\UU}{{\mathcal{U}}}
\newcommand{\VV}{{\mathcal{V}}}

\newcommand{\SO}{{\mathrm{SO}}}
\newcommand{\so}{{\mathfrak{so}}}
\newcommand{\SU}{{\mathrm{SU}}}

\newcommand{\oX}{{\overline{X}}}

\newcommand{\iu}{{\sqrt{-1}}}

\numberwithin{equation}{section}

\begin{document}

\pagenumbering{gobble} % keep title page without a number
\maketitle

\pagenumbering{arabic} % start page numbers again
\setcounter{section}{0}
\setcounter{page}{1}
\vspace{-1em}

\begin{abstract}
    We prove the representation given by a stable $\alpha_1$-cyclic parabolic $\SO_0(2,3)$-Higgs bundle satisfying specific assumptions through the non-Abelian Hodge correspondence is $\{\alpha_2\}$-almost dominated. This is a generalization of Filip's result on weight $3$ variation of Hodge structures and answers a question asked by Collier, Tholozan and Toulisse.
\end{abstract}

%~~~~~~~~~ Keywords


\large

\tableofcontents

\section{Introduction}
Higher Teichm\"uller theory, as a generalization of classical Teichm\"uller theory, is concerned with the study of representations of fundamental group $\pi_1(S)$ of oriented hyperbolic surface $S$ into simple real Lie groups $G$ of higher rank. The concept of Anosov representations introduced by F. Labourie in \cite{labourie2006anosov} plays an important role in the study of higher Teichm\"uller theory. 

Another useful tool in higher Teichm\"uller theory is the Higgs bundle. For a closed oriented hyperbolic surface $S$ equipped with a Riemann surface structure $X=(S,J)$, by the celebrated non-Abelian Hodge correspondence founded by Hitchin in \cite{hitchin1987self} and developed by Corlette, Simpson and many others, reductive representations $\pi_1(S)\to\mathrm{GL}(n,\mathbb{C})$ correspond to polystable $\mathrm{GL}(n,\mathbb{C})$-Higgs bundles, which is a holomorphic concept consisting of a holomorphic vector bundle $\EE$ with rank $n$, degree $0$ and a Higgs field $\Phi\in\mathrm{H}^0(X,\operatorname{End}(\EE)\otimes\KK_X)$, where $\KK_X$ denotes the canonical line bundle of $X$. When $G$ is a linear group, we can equip additional structure on the Higgs bundles and obtain the $G$-version non-Abelian Hodge correspondence. Moreover, there is also analogue for general real reductive Lie groups (c.f. \cite{garcia2009hitchin}). One can use the non-Abelian Hodge correspondence to deduce lots of topological properties of the character varietes of the surface group representations into the Lie group $G$. 

To get a representation from a Higgs bundle $(\EE,\Phi)$, we need to solve a PDE called the Hitchin's self-dual equation. It is with respect to the Hermitian metric $h$ on $\EE$:
\[F(\nabla^h)+[\Phi,\Phi^{*_h}]=0,\]
where $F(\nabla^h)$ denotes the curvature form of the Chern connection of $h$, and ${*_h}$ denotes the adjoint with respect to $h$. The solution gives a flat connection $\nabla^h+\Phi+\Phi^{*_h}$ and the monodromy representation $\rho$ is the desired representation. The solution metric $h$ here is called the harmonic metric and it can be illustrated as a $\rho$-equivariant harmonic map $h_\rho$ from the universal cover $\widetilde S$ of $S$ to the symmetric space of $G$. To (uniquely) solve the equation, we need the stability conditions on $(\EE,\Phi)$.

Moreover, the non-Abelian Hodge correspondence for non-compact hyperbolic surfaces of finite type has been established by Simpson, Biquard, Garc\'ia-Prada, Mundet i Riera, and many others through the study of parabolic Higgs bundles, as detailed in \cite{simpson1990harmonic} and \cite{biquard2020parabolic}. Roughly speaking, a parabolic Higgs bundle means there will be some parabolic weights at the punctures and we allow the Higgs field having some poles compatible with the weights.

In general, it is hard to check the Anosov property of a representation corresponding to a given Higgs bundle directly other than the known higher Teichm\"uller spaces or some trivial embeddings of known Anosov representations since we must solve the Hitchin's self-dual equation to get the correspondence.

In this article we will mainly focus on the case when $G=\SO_0(2,3)$. Its Lie algebra $\so(2,3)$ has two simple restricted roots $\alpha_1,\alpha_2$, where $\alpha_1$ is longer than $\alpha_2$. In \cite{collier2019geometry}, Collier, Tholozan and Toulisse considered the cyclic $\SO_0(2,3)$-Higgs bundles over a compact surface whose genus $g\geqslant2$ which can be represented by the following diagram:
\begin{equation}\label{eq:cyclic}
    \begin{tikzcd}
	{\mathcal{L}_{-2}} & {\mathcal{L}_{-1}} & {\mathcal{L}_{0}} & {\mathcal{L}_{1}} & {\mathcal{L}_{2}}
	\arrow["\tau^\vee"', from=1-1, to=1-2]
	\arrow["\beta^\vee"', from=1-2, to=1-3]
	\arrow["\beta"', from=1-3, to=1-4]
	\arrow["\tau"', from=1-4, to=1-5]
	\arrow["\gamma"', curve={height=18pt}, from=1-5, to=1-2]
	\arrow["\gamma^\vee"', curve={height=18pt}, from=1-4, to=1-1]
\end{tikzcd}
\end{equation}
with $\LL_i\cong\LL_{-i}^\vee$ and $\LL_0$ is the trivial line bundle, $\tau\colon\LL_1\to\LL_2\otimes\KK_X$ is an isomorphism and $\beta\neq0$. They can be called $\alpha_1$-cyclic Higgs bundles in the sense of \cite{labourie2017cyclic}, \cite[Section 6\& Section 7]{collier2016maximal} and \cite[Section 6]{collier2023holomorphic}. When $\beta$ is an isomorphism instead of $\tau$, such Higgs bundles have the maximal Toledo invariants and the corresponding representations are called the maximal representations. It is well-known that maximal representations are $\{\alpha_1\}$-Anosov, c.f. \cite{burger2010higher}. 

In \cite[Section 4.3]{collier2019geometry}, Collier, Tholozan, Toulisse showed that $\alpha_1$-cyclic Higgs bundles correspond to maximal fibered CFL (conformally flat Lorentz) structures on a degree $\deg(\LL_{-1})=:d$ circle bundle over $X$ whose holonomy factor through representations in the connected component of the character variety whose Toledo invariant is $2d$. By the Milnor--Wood inequality, c.f. \cite[Corollary 3.4]{burger2010surface} and \cite[Theorem 1.1]{biquard2017higgs}, we obtain that $|d|\leqslant2g-2$. When $|d|=2g-2$, these representations are maximal (hence Anosov) and when $|d|<2g-2$, these representations do not form an open domain in the representation variety. The following question was asked in \cite[Remark 4.22]{collier2019geometry}:
\begin{question}\label{question:main}
    Do $\alpha_1$-cyclic Higgs bundles give Anosov representations through the non-Abelian Hodge correspondence when $|d|<2g-2$?
\end{question}

This question is partially answered by Filip recently. In \cite{filip2021uniformization}, Filip proved it for the monodromy representations of some weight $3$ variation of the Hodge structure. Actually, it corresponds to the $\alpha_1$-cyclic Higgs bundles whose $\gamma=0$. Another markable point is that his result holds not only for compact surfaces, but also for the surfaces of finite type with a technical assumption called ``assumption A'', c.f. \cite[Definition 2.3.9]{filip2021uniformization}, with the Anosov property is changed into the relative analogue which is called ``$\log$-Anosov'', c.f. \cite[Definition 4.3.2]{filip2021uniformization}. Indeed, he proved a domination property which is equivalent to the Anosov property when the surface is compact, c.f. \cite{kapovich2018morse} and \cite{bochi2019anosov}. Since it is well-known that the Anosov property is an open condition, this also gives the Anosov property when $\gamma$ is small in some sense. However, there is no closedness for Anosov representation. Hence the Anosov property when $\gamma$ is large is still unknown.

In this article, we give a positive answer of \prettyref{question:main}:

\begin{theorem}\label{thm:coro}
    Given a compact hyperbolic Riemann surface $X$ with genus $g$. Any stable $\alpha_1$-cyclic $\SO_0(2,3)$-Higgs bundle over $X$ represented by \prettyref{eq:cyclic} with $\deg(\LL_1)<2g-2$ gives an $\{\alpha_2\}$-Anosov representation $\pi_1(X)\to\SO_0(2,3)$ through the non-Abelian Hodge correspondence. Moreover, the stability holds if and only if  $\gamma\neq0$ or $\operatorname{deg}(\LL_1)<g-1$ when $\gamma\equiv0$.
\end{theorem}

\prettyref{thm:coro} means that we construct a non-compact closed subset of Anosov representations which is unbounded in the character variety by identifying it with a family of Higgs bundles which can go to infinity in the Dolbeault moduli space. Note that when $\deg(\LL_1)=2-2g$, this also proves the $\{\alpha_2\}$-Anosov property for the Higgs bundles with vanishing quadratic differential in the $\SO_0(2,3)$-Hitchin section. When $\deg(\LL_1)=2g-2$, the corresponding representation factors through $\mathrm{O}(2,2)\times\mathrm{O}(1)$ and is $\{\alpha_1\}$-Anosov due to the maximality but not $\{\alpha_2\}$-Anosov. Please see \prettyref{rem:max} for a detailed explanation.

Similarly, our proof is also effective for non-compact surfaces and \prettyref{thm:coro} is just a special case. We generalize Filip's result with mimicking his method in the language of parabolic Higgs bundles to show some domination property of its corresponding representation. 

Now we fix a Riemann surface $X=\oX\setminus D$ equipped with a complete conformal hyperbolic metric of finite volume, where $\oX$ is a compact surface and $D$ is a finite (maybe empty) subset of $\oX$. We will also use $D$ to denote the corresponding effective divisor.

In this article, we study the $\alpha_1$-cyclic parabolic $\SO_0(2,3)$-Higgs bundles over $(\oX,D)$ (\prettyref{defn:cyclic}) and generalize Filip's ``assumption A'' (\prettyref{defn:AssumptionA}). Under our setting, an $\alpha_1$-cyclic parabolic $\SO_0(2,3)$-Higgs bundle can also be represented by the diagram \prettyref{eq:cyclic} and ``assumption A'' means a suitable choice of parabolic weights near the punctures.

Our main result is the following theorem:

\begin{theorem}\label{thm:main}
    Any stable $\alpha_1$-cyclic parabolic $\SO_0(2,3)$-Higgs bundle represented by \prettyref{eq:cyclic} satisfying assumption A of non-zero weights and $\operatorname{pardeg}(\LL_1)<\deg(\KK_\oX(D))$ gives an $\{\alpha_2\}$-almost dominated representation (\prettyref{defn:domi}) through the non-Abelian Hodge correspondence. Moreover, the stability holds if and only if $\gamma\neq0$ or $\operatorname{pardeg}(\LL_1)<\deg(\KK_{\oX}(D))/2$ when $\gamma\equiv0$.
\end{theorem}

\begin{remark}
    The condition ``non-zero weights'' is added for this positive gap. When the parabolic weight is $0$ at some punctures, we need an extra condition on $\gamma$ when following our strategy. See \prettyref{rem:zeroweight} for instance.
\end{remark}

When $X$ is compact, that is, $D=\emptyset$, the $\alpha_1$-cyclic parabolic $\SO_0(2,3)$-Higgs bundle reduces to the original $\alpha_1$-cyclic $\SO_0(2,3)$-Higgs bundle by definition. Moreover, assumption A is satisfied automatically and $\{\alpha_2\}$-almost dominated is known to be equivalent to $\{\alpha_2\}$-Anosov. Hence \prettyref{thm:coro} is a subcase of our main result \prettyref{thm:main} by taking $D=\emptyset$. 

The key point in our proof is that: the norm of $\tau$ and $\gamma$ with respect to the harmonic metric has a positive gap over the surface $X$. Although it has been known in \cite{collier2019geometry} for compact surface, we need do more careful analysis on the harmonic metric and the norm of $\gamma$ around the punctures by the model metric introduced in \cite{simpson1990harmonic}. Furthermore, this reduces to \cite[Proposition 2.2.11]{filip2021uniformization} proven by Schmid's $\mathrm{SL}_2$-orbit theorem when $\gamma\equiv0$.

To avoid discussing more on the parabolic structure on an orthogonal vector bundle of higher rank, we will mainly consider $\SO_0(2,3)$-Higgs bundles in this article. However, it is worth pointing out that our method is still effective to prove the almost-domination property for some specific parabolic $\SO_0(2,n+2)$-Higgs bundle when $n\geqslant1$. More precisely, we consider a parabolic $\SO_0(2,n+2)$-Higgs bundle has the form \prettyref{eq:cyclic} satisfying assumption A and $\operatorname{pardeg}(\LL_1)<\deg(\KK_\oX(D))$, but replace $\LL_0$ by a parabolic orthogonal vector bundle of rank $n$ with trivial determinant. When the resulting parabolic $\SO_0(2,n+2)$-Higgs bundle is stable, the corresponding representation is $\{\alpha_2\}$-almost-dominated. See \prettyref{rem:higher} for an explanation.

\paragraph{Structure of the article} We will give some preliminaries on Anosov representations and parabolic Higgs bundles in \prettyref{sec:pre}. The definition and some properties of $\alpha_1$-cyclic parabolic $\SO_0(2,3)$-Higgs bundles will be given in \prettyref{sec:Higgsestimate}. In \prettyref{sec:estimates}, we recall Filip's estimates on a certain class of Morse functions. Finally, we give the proof of our main result in \prettyref{sec:main}.

\paragraph{Acknowledgements}
The author is grateful to his advisor Qiongling Li for suggesting this problem and many helpful discussions. The author also thanks Brian Collier and Zachary Virgilio for comments on an earlier version of this paper. The author wishes to thank the referee for their valuable comments, which helped to improve the exposition substantially. The author is partially supported by the National Key R\&D Program of China No. 2022YFA1006600, the Fundamental Research Funds for the Central Universities and Nankai Zhide Foundation.

\section{Preliminaries}\label{sec:pre}

\subsection{Lie Theory Background}\label{sec:Lie}

We recommend \cite[Chapter \uppercase\expandafter{\romannumeral6}]{knapp1996lie} for the Lie theory background. Let $G$ be is a semisimple real Lie group with Lie algebra $\mathfrak{g}:=\operatorname{Lie}(G)$ and the exponential map $\exp\colon\mathfrak{g}\to G$. We fix a maximal subgroup $K$ of $G$ and let its Lie algebra be $\mathfrak{k}:=\operatorname{Lie}(K)$. This gives Cartan involutions $\Theta_G\colon G\to G$ and $\Theta_\mathfrak{g}\colon \mathfrak{g}\to\mathfrak{g}$ on both Lie group level and Lie algebra level such that $K$ and $\mathfrak{k}$ are the fixed points of $\Theta_G$ and $\Theta_\mathfrak{g}$ respectively. Now the eigenspaces decomposition of the Cartan involution $\Theta_\mathfrak{g}$ gives the Cartan decomposition $\mathfrak{g}=\mathfrak{k}\oplus\mathfrak{p}$, where $\mathfrak{p}$ is the $(-1)$-eigenspace of $\Theta_\mathfrak{g}$.

We take a maximal Abelian subspace $\mathfrak{a}$ of $\mathfrak{p}$. The adjoint action of $\mathfrak{a}$ on $\mathfrak{g}$ gives a weight space decomposition
\[\mathfrak{g}=\mathfrak{g}_0\oplus\bigoplus_{\alpha\in\Phi}\mathfrak{g}_{\alpha},\]
where $\Phi\subset\mathfrak{a}^{\vee}$ is the set of \textbf{restricted roots} of $\mathfrak{g}$ with respect to $\mathfrak{a}$.

We fix a set of positive roots $\Phi^{+}\subset\Phi$, i.e. $\Phi^+$ is contained in a half-space of $\mathfrak{a}^\vee$ and $\Phi=\Phi^+\coprod\Phi^-$, where $\Phi^-=-\Phi^+$. Let $\Delta\subset\Phi^+$ denote the corresponding set of simple roots. The associated \textbf{closed positive Weyl chamber} $\overline{\mathfrak{a}^+}\subset\mathfrak{a}$ is defined as
\[\overline{\mathfrak{a}^+}= \{v\in\mathfrak{a}\mid\alpha(v)\geqslant 0, \forall\alpha\in\Delta\}.\]

There is a decomposition of $G$ called KAK decomposition as a generalization of singular value decomposition. Explicitly, for any $g\in G$, there exists a unique $\mu(g)\in\overline{\mathfrak{a}^+}$ such that there exist two elements $k_{-}(g), k_+(g)\in K$ satisfying 
\[g=k_-(g)\exp(\mu(g))k_+(g).\]
Moreover, if there is $k_-^\prime(g),k_+^\prime(g)$ such that $g=k_-^\prime(g)\exp(\mu(g))k_+^\prime(g)$, then there is an $m\in K$ commutes with $\exp(\mu(g))$ such that $k_-^\prime(g)=k_-(g)m$ and $k_+^\prime(g)=m^{-1}k_+(g)$. The well-defined map
\[\mu\colon G\to\overline{\mathfrak{a}^+}\]
is called the \textbf{Cartan projection} of $G$.

Since the analytic Weyl group $W(G,A)$ acts freely and transitively on the Weyl chambers, there exists an element $k^{op}\in K$ such that for any $g\in G$ decomposes as $k_-(g)\exp(\mu(g))k_+(g)$, the KAK decomposition of $g^{-1}$ can be given by
\[g^{-1}=\left((k_+(g))^{-1}(k^{op})^{-1}\right)\exp\left(\operatorname{Ad}(k^{op})(-\mu(g))\right)\left(k^{op}(k_-(g))^{-1}\right),\]
where $\operatorname{Ad}\colon G\to\mathrm{GL}(\mathfrak{g})$ denotes the adjoint action of $G$. In other words, we have
\[\mu(g^{-1})=\operatorname{Ad}(k^{op})(-\mu(g)).\]
The map $\iota^{op}\colon\mu\mapsto\operatorname{Ad}(k^{op})(-\mu)$ is called the \textbf{opposition involution} on $\overline{\mathfrak{a}^+}$.

\begin{example}[Restricted Root System of $\so(p,q)$ for $p<q$]\label{example:sopq}
    We use the standard non-degenerate bilinear form
$$\begin{aligned}
Q\colon \mathbb{R}^{p+q}\times\mathbb{R}^{p+q}&\longrightarrow\mathbb{R}\\
\left(\begin{pmatrix}
x_1\\ \vdots\\x_{p+q}\end{pmatrix},\begin{pmatrix}
y_1\\ \vdots\\y_{p+q}\end{pmatrix}
\right)&\longmapsto \sum_{i=1}^px_iy_i-\sum_{j=1}^qx_{p+j}y_{p+j}
\end{aligned}$$
of signature $(p,q)$ to get the group
$$\begin{aligned}\mathrm{SO}(p,q)&=\left\{A\in \mathrm{SL}(p+q,\mathbb{R})\mid Q(x,y)=Q(Ax,Ay),\forall x,y\in\mathbb{R}^{p+q}\right\}\\&=\left\{A\in\mathrm{SL}(p+q,\mathbb{R})\mid A^{\mathrm{t}}I_{p,q}A=I_{p,q}\right\},\end{aligned}$$
where $$I_{p,q}=\begin{pmatrix}
    I_p& 0\\0& -I_q
\end{pmatrix}.$$
Then $\mathrm{SO}_0(p,q)$ is defined as the identity component of $\mathrm{SO}(p,q)$. Its Lie algebra is
\[\begin{aligned}
    \so(p,q):=&\operatorname{Lie}(\mathrm{SO}_0(p,q))=\{A\in\mathfrak{sl}(p+q,\mathbb{R})\mid A^{\mathrm{t}}I_{p,q}+I_{p,q}A=0\}\\
    =&\left\{\begin{pmatrix}
        A_{11}& A_{12}\\A_{21}& A_{22}
    \end{pmatrix}\in\mathfrak{sl}(p+q,\mathbb{R})\Bigg| A_{11}+A_{11}^{\mathrm{t}}=0,A_{22}+A_{22}^{\mathrm{t}}=0,A_{21}=A_{12}^{\mathrm{t}}\right\}
\end{aligned}\]
Below we denote that $G=\SO_0(p,q)$, $\mathfrak{g}=\so(p,q)$. We fix $K=\SO(p)\times\SO(q)$ as the maximal compact subgroup of $G$, and $\mathfrak{k}:=\operatorname{Lie}(K)=\so(p)\oplus\so(q)$. Thus the Cartan decomposition of $\mathfrak{g}$ can be expressed as
\[\begin{tikzcd}
	{\mathfrak{g}} & {\mathfrak{k}} & {\mathfrak{p}} \\
	{\begin{pmatrix}         A_{11}& A_{12}\\ A_{21}& A_{22}     \end{pmatrix}} & {\begin{pmatrix}         A_{11}& 0\\0& A_{22}     \end{pmatrix}} & {\begin{pmatrix}         0& A_{12}\\A_{21}& 0     \end{pmatrix}}
	\arrow["\in"{marking}, draw=none, from=2-1, to=1-1]
	\arrow["\oplus"{marking, pos=0.5}, draw=none, from=1-2, to=1-3]
	\arrow["\in"{marking}, draw=none, from=2-2, to=1-2]
	\arrow["{+}"{marking, pos=0.6}, draw=none, from=2-2, to=2-3]
	\arrow["\in"{marking}, draw=none, from=2-3, to=1-3]
	\arrow["{=}"{marking, pos=0.25}, draw=none, from=1-1, to=1-2]
	\arrow["{=}"{marking, pos=0.5}, draw=none, from=2-1, to=2-2]
\end{tikzcd}\]
where $A_{11}$ and $A_{22}$ are skew-symmetric real matrices and $A_{12}=A_{21}^{\mathrm{t}}$ is a real $(p\times q)$-matrix. Now we take 
\[\mathfrak{a}=\left\{A=\begin{pmatrix}
    0& A_{12}\\ A_{21}& 0
\end{pmatrix}\in\mathfrak{p}\left| A_{12}=\begin{pmatrix}
0&\cdots&    0& a_1& 0&\cdots&0\\
0&\cdots&    a_2& 0&0&\cdots&0\\
0&\iddots&0&\cdots&\cdots&\cdots&0\\
a_p&0&\cdots&\cdots&\cdots&\cdots&0
\end{pmatrix}, A_{21}=A_{12}^{\mathrm{t}}\right.\right\}.\]
Let $\theta_i\in\mathfrak{a}^\vee$ be the linear functions such that $\theta_i(A)=a_i$ for $i=1,\dots,p$. The corresponding restricted roots are
\[\Phi=\{\pm \theta_i\pm \theta_j\mid 1\leqslant i<j\leqslant p\}\cup\{\pm \theta_i\mid1\leqslant i\leqslant p\}.\]
We choose $\Delta=\{\alpha_i:=\theta_i-\theta_{i+1}\mid 1\leqslant i\leqslant p-1\}\cup\{\alpha_p:=\theta_p\}$ as the simple roots. The closed positive Weyl chamber is
\[\overline{\mathfrak{a}^+}=\left\{\begin{pmatrix}
    0& A_{12}\\ A_{21}& 0
\end{pmatrix}\in\mathfrak{a}\left| A_{12}=\begin{pmatrix}
0&\cdots&    0& a_1& 0&\cdots&0\\
0&\cdots&    a_2& 0&0&\cdots&0\\
0&\iddots&0&\cdots&\cdots&\cdots&0\\
a_p&0&\cdots&\cdots&\cdots&\cdots&0
\end{pmatrix}, a_1\geqslant\cdots\geqslant a_p\geqslant0\right.\right\}.\]
We have $\mu(g^{-1})=\mu(g)$.
\end{example}

\subsection{Anosov Representations and Almost-domination}
Any Riemann surface whose universal cover isomorphic to the upper half-plane $\mathbb{H}^2$ can be equipped with a unique complete conformal hyperbolic metric $g_{\mathrm{hyp}}$, i.e. the hyperbolic metric descended from the Poincar\'e metric on $\mathbb{H}^2$. Now let $\oX$ be a compact Riemann surface, $D\subset\oX$ be a possibly empty finite set of points. We will also denote by $D$ the corresponding effective divisor over $\overline{X}$. Let $X := \oX\setminus D$ be the corresponding punctured Riemann surface with its canonical line bundle $\KK_X$. Assume
that the Euler characteristic $\chi(X)$ of $X$ is negative. Then its universal cover is isomorphic to $\mathbb{H}^2$ and it can be equipped with a unique complete conformal hyperbolic metric $g_{\mathrm{hyp}}$. Fix a basepoint $x_0\in X$ such that with respect to the universal cover $\pi\colon \mathbb{H}^{2}\cong\widetilde{X}\to X$, $x_0$ can be lifted to $\iu=\widetilde{x_0}\in\mathbb{H}^2$. 

\begin{example}[singularity of the hyperbolic metric]\label{example:cusphyperbolic}
    For the punctured unit disk
    \[\mathbb{D}^*:=\{z\in\mathbb{C}\mid 0<|z|<1\},\]
    its universal cover is 
    \[\begin{aligned}
        \pi\colon\mathbb{H}^2\longrightarrow&\mathbb{D}^*\\
        w\longmapsto&\exp(2\pi\iu w).
    \end{aligned}\]
    The hyperbolic metric on $\mathbb{H}^2$ is $\frac{|\dd w|^2}{(\operatorname{Im}w)^2}$ and it descends to $\frac{|\dd z|^2}{(|z|\ln|z|)^2}$ on $\mathbb{D}^*$.
\end{example}

Let $\rho_{F}\colon\pi_1(X)\to\mathrm{PSL}(2,\mathbb{R})$ denote the Fuchsian representation coming from the hyperbolic metric $g_{\mathrm{hyp}}$ on $X$. For an element $\sigma\in\pi_1(X)$, we will use $\|\sigma\|$ to denote the matrix norm of $\rho_F(\sigma)$ (after choosing one of its matrix representation in $\mathrm{SL}(2,\mathbb{R})$.). One can easily check that if we identify $(\widetilde{X},\widetilde{x_0})$ with $(\mathbb{H}^2,\iu)$, then
\[\|\sigma\|=\sqrt{2\cosh(d(\widetilde{x_0},\widetilde{x_0}\cdot\sigma))},\]
where $d$ denotes the distance function on $(\widetilde{X},g_{\mathrm{hyp}})$.

The notion of almost-dominated representation was introduced in \cite{zhu2021relatively} for general relatively hyperbolic group. Here we regard $\pi_1(X)$ naturally as the relatively hyperbolic group (relative to its cusp subgroup) arising from $g_{\mathrm{hyp}}$.

\begin{definition}\label{defn:domi}
    For a fixed subset of simple restricted roots $\theta\subset\Delta$ of a semisimple real Lie group $G$, a representation $\rho\colon\pi_1(X)\to G$ is called \textbf{$\theta$-almost dominated} if there exist $C,\varepsilon > 0$ such that
    \[\alpha\left(\mu(\rho(\sigma))\right)\geqslant \varepsilon\cdot\ln\|\sigma\|-C,\forall\alpha\in\theta,\sigma\in\pi_1(X),\]
    where $\mu$ denotes the Cartan projection. Or equivalently, a representation $\rho$ is $\theta$-almost dominated if there exist $C,\varepsilon > 0$ such that \[\alpha\left(\mu(\rho(\gamma))\right)\geqslant \varepsilon\cdot d(\widetilde{x_0},\widetilde{x_0}\cdot\sigma)-C,\forall\alpha\in\theta,\sigma\in\pi_1(X).\]
\end{definition}

We have the following equivalence when the surface is compact:

\begin{fact}[\cite{kapovich2018morse,bochi2019anosov}]\label{fact:equi}
    When $X$ is compact, a representation $\rho:\pi_1(X)\to G$ is $\theta$-almost dominated if and only if it is $\theta$-Anosov.
\end{fact}

\begin{remark}
    By the Milnor--\v{S}varc lemma, $\pi_1(X)$ equipped with the word length is quasi-isometric to $\widetilde{X}$ when $X$ is compact. Therefore, the notion of almost-domination does not depend on the choice of complex structure and can be expressed in terms of the word length on $\pi_1(X)$.
\end{remark}

\subsection{Parabolic Higgs Bundles}

Recall that $\oX$ is a compact Riemann surface with a finite subset $D$ (maybe empty) on it satisfying that $X:=\oX\setminus D$ has negative Euler characteristic. We fix a real semisimple Lie group $G$ with its Cartan decomposition $\mathfrak{g}=\mathfrak{k}\oplus\mathfrak{p}$. Let $K^\mathbb{C}$ be the complexification of $K$ and $\mathfrak{g}^{\mathbb{C}}=\mathfrak{k}^{\mathbb{C}}\oplus\mathfrak{p}^{\mathbb{C}}$ be the complexified Cartan decomposition. Below we freely use the notations in \prettyref{sec:Lie} and the following (vector or principal) bundles are all holomorphic. 

Suppose $M$ is a $K^{\mathbb{C}}$-set, i.e. $K^{\mathbb{C}}$ has a left action on it, then we can define the associated bundle
$$\mathbb{E}[M]=\mathbb{E}\times_{K^{\mathbb{C}}}M:=\left(\mathbb{E}\times M\right)/K^{\mathbb{C}},$$
where the $K^{\mathbb{C}}$-action on $\mathbb{E}\times M$ is
$$\begin{aligned}
    K^{\mathbb{C}}\times(\mathbb{E}\times M)&\longrightarrow \mathbb{E}\times M\\
    (k,(e,m))&\longmapsto (e\cdot k^{-1},k\cdot m).
\end{aligned}$$

The concept of parabolic $G$-Higgs bundle over $(\oX,D)$ was introduced by O. Biquard, O. Garc\'{i}a-Prada and I. M. i Riera in \cite{biquard2020parabolic}. By their definition, a parabolic $G$-Higgs bundle over $(\oX,D)$ consists of the following data:
\begin{itemize}
    \item[(1)] a parabolic principal $K^\mathbb{C}$-bundle $\mathbb{E}$ with parabolic structure $(Q_j,\zeta_j)$ at each $x_j\in D$;

    \item[(2)] a parabolic $G$-Higgs field $\Phi\in\mathrm{H}^0(X,\mathbb{E}[\mathfrak{p}^{\mathbb{C}}]\otimes\KK_X)$ with singularities of certain type around $D$, where $K^\mathbb{C}$ acts on $\mathfrak{p}^{\mathbb{C}}$ via the isotropic representation which is restricted from the adjoint action $K^\mathbb{C}\to\operatorname{Ad}(\mathfrak{g}^{\mathbb{C}})$.
\end{itemize}

In this section we will use the definition of parabolic $G$-Higgs bundle introduced in \cite[Section 2 \& Section 4]{biquard2020parabolic} for general real reductive Lie group $G$ and illustrate the \textbf{parabolic $G$-Higgs bundle} and the \textbf{stability condition} from the viewpoint of vector bundles for $G=\mathrm{SL}(n,\mathbb{C})$ and our case $G=\SO_0(p,q)$. 

We first define the following notations.

\begin{definition}\label{defn:flag}
            Suppose $V$ is a $\mathbb{C}$-linear space. A sequence of subspaces of $V$
            \[0=F_{k}\subsetneq F_{k-1}\subsetneq\cdots\subsetneq F_2 \subsetneq F_{1}=V, \quad(\mbox{resp. }0=F_{1}\subsetneq F_{2}\subsetneq\cdots \subsetneq F_{k-1}\subsetneq F_{k}=V)\]
            is called a \textbf{reverse flag} (resp. \textbf{flag}). If $V$ is equipped with a bilinear form $Q$, then the above reverse flag (resp. flag) is called a \textbf{reverse isotropic flag} (resp. \textbf{isotropic flag}) if every $F_i$ is isotropic or coisotropic under $Q$ and $F_i=(F_{k+1-i})^{\perp_Q}$.
            
            We say a reverse flag $\left(F_i\right)_{i=1}^k$ is \textbf{equipped with decreasing real numbers $\left(\zeta_j\right)_{j=1}^{\dim(V)}$} if $\zeta_j\geqslant\zeta_{j+1}$, 
            \[\zeta_{\dim(V)-\dim(F_i)+1}=\zeta_{\dim(V)-\dim(F_i)+2}=\cdots=\zeta_{\dim(V)-\dim(F_{i+1})}=:\widetilde{\zeta}_i\]
            and $\widetilde{\zeta}_i>\widetilde{\zeta}_{i+1}$
            for any $i=1,\dots,k-1$. For our convenience, we usually set $\widetilde{\zeta}_0=\zeta_0=\widetilde{\zeta}_{\dim(V)+1}=\zeta_{\dim(V)+1}=0$. $\operatorname{Gr}_{\widetilde{\zeta}_i}(V):=F_i/F_{i+1}$ are called the \textbf{graded pieces} of $V$.
            
            A basis $\{e_1,\dots,e_{\dim(V)}\}$ of $V$ is called \textbf{compatible with a reverse flag $\left(F_i\right)_{i=1}^{k}$} if \[e_{\dim(V)-\dim(F_{i})+1},\dots,e_{\dim(V)}\] span $F_i$ for any $i=1,\dots,k-1$.
        \end{definition}

For our convenience, the parabolic weights will be chosen in a smaller set than the set $\iu\overline{\mathcal{A}}$ in the definition of parabolic principal bundle in \cite[Section 2.1]{biquard2020parabolic} to avoid the parahoric objects mentioned in \cite[Section 3]{biquard2020parabolic}. A similar setting was adopted in \cite{feng2023compact} when $G=\SO_0(2,q)$.

\paragraph{Parabolic $\mathrm{SL}(n,\mathbb{C})$-Higgs Bundle}

When $G=\mathrm{SL}(n,\mathbb{C})$, we take $K=\SU(n)$, then $K^\mathbb{C}$ is also $\mathrm{SL}(n,\mathbb{C})$. For a parabolic $\mathrm{SL}(n,\mathbb{C})$-Higgs bundle $(\mathbb{E},\Phi)$, from the viewpoint of the vector bundle $\mathbb{E}[\mathbb{C}^n]$ induced by the standard action, we obtain that a parabolic $\mathrm{SL}(n,\mathbb{C})$-Higgs bundle is equivalent to the following data:
\begin{itemize}
    \item[(1)] a holomorphic vector bundle $\EE\to\oX$, with $\operatorname{rank}(\EE)=n$ and $\det(\EE)=\mathcal{O}_{\oX}$ which is the trivial line bundle over $\oX$;

    \item[(2)] a reverse flag $\left(\EE_i^j\right)$ of $\EE_{x_j}$ equipped with decreasing real numbers $\left(\zeta_k^j\right)_{1\leqslant k\leqslant n}$ (called \textbf{parabolic weights}) satisfying that 
    \[\sum_{k=1}^n\zeta_k^j=0\mbox{\quad and\quad }\zeta_k^j-\zeta_l^j\in[0,1),\forall 1\leqslant l<k\leqslant n\]
    for every marked point $x_j\in D$; 
    
    \item[(3)]
    a Higgs field $\Phi$ which is a holomorphic section of $\operatorname{End}(\mathcal{E})\otimes\mathcal{K}(D)$ such that $\operatorname{tr}(\Phi)=0$ and with respect to a coordinate chart $(U,z)$ centered at $x_j$, a holomorphic frame $\{e_1,\dots,e_n\}$ compatible with the reverse flag $\left(\mathcal{E}_i^j\right)$
\begin{equation}\label{eq:Higgsfield}
        \Phi=\sum_{k=1}^n\sum_{l=1}^nO\left(z^{\left\lceil\zeta_k^j-\zeta_l^j\right\rceil}\right)\cdot e_k\otimes e_l^\vee\otimes\dfrac{\dd z}{z}.
\end{equation}
which means $\Phi$ is allowed to have a pole of order $1$ at $x_j$ along $e_k\otimes e_l^\vee$ when $\zeta_k^j\leqslant\zeta_l^j$ and must be holomorphic at $x_j$ along $e_k\otimes e_l^\vee$ when $\zeta_k^j>\zeta_l^j$.
\end{itemize}

\begin{remark}
    This convention may be confusing when compared to the well-known definition of filtered regular Higgs bundle in \cite{simpson1990harmonic} or parabolic Higgs bundle in \cite{yokogawa1993compactification}. When using the well-known definition of parabolic Higgs bundle, the underlying vector bundle may not have trivial determinant. Here we just apply a translation on the weights to make the underlying vector bundle have trivial determinant. 
\end{remark}

Now an automorphism of a parabolic $\mathrm{SL}(n,\mathbb{C})$-Higgs bundle $(\EE,\EE_i^j,\zeta_k^j,\Phi)$ is an automorphism of $\EE$ which stablizes $\Phi$ and preserves the reverse flag $\EE_i^j$.

The parabolic degree defined below will be used to test the stability condition.

\begin{definition}
            For any holomorphic subbundle $\mathcal{E}^\prime$ of a parabolic $\mathrm{SL}(n,\mathbb{C})$-Higgs bundle $(\mathcal{E},\mathcal{E}_{i}^j,\zeta_k^j,\Phi)$, we define the \textbf{parabolic degree} of $\mathcal{E}^\prime$ as
    \[\operatorname{pardeg}(\mathcal{E}'):=\deg(\mathcal{E}')-\sum_{j=1}^{\deg(D)}\sum_{i}(\widetilde{\zeta}_{i}^j-\widetilde{\zeta}_{i-1}^j)\dim\left(\left(\mathcal{E}'\right)_{x_j}\cap\mathcal{E}_i^j\right),\]
    where we assume $\widetilde{\zeta}_{0}^j=0$.
\end{definition}

\begin{remark}
    Note that here we use ``$-$'' connecting the degree part and the parabolic part instead of ``$+$'' appearing in the well-known definition since we use the reverse flag and decreasing weights.
\end{remark}

\begin{definition}
    A parabolic $\mathrm{SL}(n,\mathbb{C})$-Higgs bundle $(\EE,\Phi)$ is \textbf{semistable} if for any proper holomorphic subbundle $\EE^\prime\subset\EE$ which is $\Phi$-invariant, $\operatorname{pardeg}(\EE')\leqslant0$. Furthermore, it is \textbf{stable} if the above inequality is strict when $\EE'\neq0,\EE$.
\end{definition}

\paragraph{Parabolic $\mathrm{SO}_0(p,q)$-Higgs bundle}

When $G=\SO_0(p,q)$, we take \[K=\SO(p)\times\SO(q),\] then $K^\mathbb{C}=\mathrm{SO}(p,\mathbb{C})\times\mathrm{SO}(q,\mathbb{C})$. For a parabolic $\SO_0(p,q)$-Higgs bundle $(\mathbb{E},\Phi)$, from the viewpoint of the vector bundle $\mathbb{E}[\mathbb{C}^p\oplus\mathbb{C}^q]$ induced by the standard action, we obtain that a parabolic $\SO_0(p,q)$-Higgs bundle is equivalent to the following data:
\begin{itemize}
    \item[(1)] a holomorphic vector bundle: $\EE=\UU\oplus\VV$ over $\oX$, where $\operatorname{rank}(\UU)=p$ and $\operatorname{rank}(\VV)=q$ with $\det(\UU)=\det(\VV)=\mathcal{O}_{\oX}$ which is the trivial line bundle over $\oX$;

    \item[(2)] two holomorphic symmetric non-degenerate bilinear forms $Q_\UU\colon\operatorname{Sym}^2(\UU)\to\mathcal{O}$ and $Q_\VV\colon\operatorname{Sym}^2(\VV)\to\mathcal{O}$ with the induced isomorphisms $q_\UU\colon\UU\to\UU^\vee$ and $q_\VV\colon\VV\to\VV^\vee$;

    \item[(3)] a reverse isotropic flag $\left(\UU_s^j\right)$ (resp. $\left(\VV_t^j\right)$) of $\UU_{x_j}$ (resp. $\VV_{x_j}$) with respect to $Q_\UU$ (resp. $Q_\VV$) equipped with decreasing real numbers $\left(\zeta_k^j\right)_{1\leqslant k\leqslant p}$ and (resp. $\left(\eta_l^j\right)_{1\leqslant l\leqslant q}$) (called \textbf{parabolic weights}) satisfying that $\zeta_k^j,\eta_l^j\in(-1/2,1/2)$ and 
    \[\zeta_k^j+\zeta_{p+1-k}^j=0,\qquad \eta_l^j+\eta_{q+1-l}^j=0,\]
    for every marked point $x_j\in D$; 
    
    \item[(4)]
    a Higgs field $\Phi=\begin{pmatrix}
    0&\gamma\\
    -\gamma^*&0
\end{pmatrix}$ with respect to the decomposition $\EE=\UU\oplus\VV$ which is a holomorphic section of $\operatorname{End}(\mathcal{E})\otimes\mathcal{K}(D)$, where $\gamma^*=q_{\VV}^{-1}\circ\gamma^{\vee}\circ q_{\UU}$, such that with respect to a coordinate chart $(U,z)$ centered at $x_j$, a holomorphic frame $\{e_1,\dots,e_p\}$ (resp. $\{f_1,\dots,f_q\}$) compatible with the reverse isotropic flag $\left(\mathcal{U}_s^j\right)$ (resp. $\left(\VV_t^j\right)$), \[\gamma=\sum_{k=1}^p\sum_{l=1}^qO\left(z^{\left\lceil\zeta_k^j-\eta_l^j\right\rceil}\right)\cdot e_k
\otimes f_l^\vee\otimes\dfrac{\mathrm{d}z}{z},\]
which means $\gamma$ is allowed to have a pole of order $1$ at $x_j$ along $e_k\otimes f_l^\vee$ when $\zeta_k^j\leqslant\eta_l^j$ and must be holomorphic at $x_j$ along $e_k\otimes f_l^\vee$ when $\zeta_k^j>\eta_l^j$.
\end{itemize}

Therefore, a parabolic $\SO_0(p,q)$-Higgs bundle can be viewed as a parabolic $\mathrm{SL}(p+q,\mathbb{C})$-Higgs bundle naturally. But the stability condition for parabolic $\SO_0(p,q)$-Higgs bundles are different.

\begin{definition}\label{defn:stable}
    A parabolic $\mathrm{SO}_0(p,q)$-Higgs bundle $(\UU\oplus\VV,Q_{\UU},Q_{\VV},\Phi)$ ($2\leqslant p<q$) is \textbf{semistable} if for any isotropic subbundle $\UU^\prime\oplus\VV^\prime\subset\EE$ which is $\Phi$-invariant, $\operatorname{pardeg}(\UU^\prime\oplus\VV^\prime)\leqslant0$. Furthermore, it is \textbf{stable} if the above inequality is strict when
    \[\begin{cases}\UU^\prime\oplus\VV^\prime\neq0& \mbox{ when }p>2,\\
        \VV^\prime\neq0& \mbox{ when }p=2.
    \end{cases}\]
\end{definition}

\begin{remark}
    The difference between $p=2$ and $p>2$ arises because $\SO(2,\mathbb{C})$ has no non-trivial parabolic subgroup. Hence $\SO(2,\mathbb{C})$ would not give any proper reduction. Please see \cite[Proposition 2.16]{aparicio2019so}. One could also find a detailed argument to obtain the stability condition of parabolic $\SO_0(2,q)$-Higgs bundles via vector bundle in \cite[Section 3]{feng2023compact}.
\end{remark}

\subsection{Hitchin--Kobayashi Correspondence}\label{sec:HK}

The Hitchin--Kobayashi correspondence, which associates a harmonic metric and a representation $\rho\colon\pi_1(X)\to G$ to (poly-)stable parabolic $G$-Higgs bundle, was proven in \cite{simpson1990harmonic} for filtered regular Higgs bundles and in \cite[Theorem 5.1 \& Section 6.1]{biquard2020parabolic} for general parabolic $G$-Higgs bundles. In this section, we will explain what it means for $G=\mathrm{SL}(n,\mathbb{C})$ and $\SO_0(p,q)$. 

We first explain what the metric means for a parabolic $G$-Higgs bundle when $G=\mathrm{SL}(n,\mathbb{C})$ and $\SO_0(p,q)$. Recall that $K$ denotes the maximal compact subgroup of $G$ and $K^\mathbb{C}$ denotes its complexification. For a parabolic $G$-Higgs bundle $(\mathbb{E},Q_j,\zeta_j,\Phi)$, a metric $h$ on the parabolic principal $K^\mathbb{C}$-bundle $\mathbb{E}$ means a global smooth section of $\mathbb{E}\left[K\backslash K^{\mathbb{C}}\right]$. 

Let the standard basis of $\mathbb{C}^n$ be $(\varepsilon_i)$. Below we will call a basis $(e_i)$ of $\mathbb{C}^n$ is a \textbf{unit-basis} if $e_1\wedge\cdots\wedge e_n=\varepsilon_1\wedge\cdots\wedge\varepsilon_n$. Similarly, for a vector bundle $\EE$ whose determinant bundle is trivial, a frame $(e_i)$ of $\EE$ is called a \textbf{unit-frame} if $e_1\wedge\cdots\wedge e_n$ is the standard section of $\mathcal{O}\cong\det(\EE)$. 

Any Hermitian metric $h$ with standard volume on $\mathbb{C}^n$ corresponds to a positive definite Hermitian matrix $\left(h(\varepsilon_k,\varepsilon_l)\right)_{1\leqslant k,l\leqslant n}$ with determinant $1$. Hence there is a natural correspondence given by
\[\begin{aligned}
    \SU(n)\backslash\mathrm{SL}(p,\mathbb{C})&\longrightarrow\{\mbox{Hermitian metric with standard volume on }\mathbb{C}^n\}\\
    \SU(p)\cdot g&\longmapsto \overline{g^\mathrm{t}}g.
\end{aligned}\]
Therefore, the metric $h$ on a parabolic $\mathrm{SL}(n,\mathbb{C})$-Higgs bundle $(\mathbb{E},\Phi)$ corresponds to an Hermitian metric $h_{\EE}$ on $\EE=\mathbb{E}[\mathbb{C}^n]$ whose induced metric on $\det(\EE)\cong\mathcal{O}$ is the standard trivial metric. 

For a parabolic $\SO_0(p,q)$-Higgs bundle $(\UU,\VV,Q_\UU,Q_\VV,\Phi)$, the harmonic metric $h$, which is a global smooth section of \[\mathbb{E}[(\SO(p)\times\SO(q))\backslash(\SO(p,\mathbb{C})\times\SO(q,\mathbb{C}))],\]
corresponds to two Hermitian metrics $h_{\UU}$ and $h_{\VV}$ on $\UU$ and $\VV$ which are compatible with the orthogonal structures $Q_{\UU}$ and $Q_{\VV}$ in the following sense:
\[h_{\UU}=(q_{\UU})^*(h_{\UU}^\vee),h_{\VV}=(q_{\VV})^*(h_{\VV}^\vee),\]
where $h_{\UU}^\vee$ and $h_{\VV}^\vee$ are the dual metric defined on $\UU^\vee$, $\VV^\vee$ respectively and the harmonic metric on $\EE=\UU\oplus\VV$ is just $h_\EE=h_\UU\oplus h_\VV$.

Now we define the \textbf{model metric} for a parabolic $\mathrm{SL}(n,\mathbb{C})$-Higgs bundle $(\EE,\EE_i^j,\zeta_k^j,\Phi)$. This concept was introduced in \cite[Section 5]{simpson1990harmonic} for filtered regular Higgs bundles and generalized in \cite[Section 5.1]{biquard2020parabolic} for general parabolic $G$-Higgs bundles. We recommend \cite{kim2018analytic} and \cite[Section 5.1]{biquard2020parabolic} as references here.

For a local coordinate chart $(U,z)$ centered at a puncture $x_j\in D$, the Higgs field $\Phi$ can be represented as $\phi(z)\dd z/z$, where $\phi$ is a holomorphic endomorphism of $\EE|_U$. The \textbf{residue} of $\Phi$ at $x_j$ is defined as 
\[\operatorname{Res}_{x_j}\Phi:=\phi(0).\]
By \prettyref{eq:Higgsfield}, we obtain that $\operatorname{Res}_{x_j}\Phi$ preserves the reverse flag $\left(\EE_i^j\right)$. Denote the induced endomorphism of $\operatorname{Res}_{x_j}\Phi$ on the graded pieces $\bigoplus_{\delta\in\mathbb{R}}\operatorname{Gr}_\delta(\EE_{x_j})$ (\prettyref{defn:flag}) by the \textbf{graded residue} $\operatorname{GrRes}_{x_j}\Phi$. The graded piece further splits as a direct sum 
\[\operatorname{Gr}_\delta(\EE_{x_j})=\bigoplus_{\lambda\in\mathbb{C}}\operatorname{Gr}_{\delta}^\lambda(\EE_{x_j})\]
with respect to the generalized eigenspaces of $\operatorname{GrRes}_{x_j}\Phi$. We take the nilpotent part $Y_\delta$ (with $Y_{\delta}^{m+1}=0$ for some $m\geqslant0$) of the graded residue on each $\operatorname{Gr}_\delta(\EE_{x_j})$. 

The $Y_\delta$ then induces a (unique) further filtration $\{W_r\operatorname{Gr}^\lambda_\delta(\EE_{x_j})\}_{r\in\mathbb{Z}}$ called the \textbf{weight filtration} with corresponding grading
\begin{equation}\label{eq:weightfiltration}
\operatorname{Gr}_\delta(\EE_{x_j})=\bigoplus_{r\in\mathbb{Z}}\bigoplus_{\lambda\in\mathbb{C}}\operatorname{Gr}_r\operatorname{Gr}_\delta^\lambda(\EE_{x_j}),
\end{equation}
where $\operatorname{Gr}_r=W_r/W_{r-1}$, which satisfy that
\begin{itemize}
    \item[(1)] $0\subset W_{-m}\operatorname{Gr}^\lambda_\delta(\EE_{x_j})\subset\cdots\subset W_m\operatorname{Gr}^\lambda_\delta(\EE_{x_j})=\operatorname{Gr}^\lambda_\delta(\EE_{x_j})$;

    \item[(2)] $Y_{\delta}(W_r\operatorname{Gr}^\lambda_\delta(\EE_{x_j}))\subset W_{r-2}\operatorname{Gr}^\lambda_\delta(\EE_{x_j})$;

    \item[(3)] $Y_{\delta}^r$ induces an isomorphism between $\operatorname{Gr}_{r}\operatorname{Gr}^\lambda_\delta(\EE_{x_j})$ and $\operatorname{Gr}_{-r}\operatorname{Gr}^\lambda_\delta(\EE_{x_j})$.
\end{itemize}

Therefore if we define the diagonal endomorphism 
\[H_\delta=\sum_{r\in\mathbb{Z}}r\cdot \operatorname{id}_{\operatorname{Gr}_r\operatorname{Gr}_{\delta}^\lambda(\EE_{x_j})},\]
then $[H_\delta, Y_\delta] =-2Y_\delta$. Then there exists an endomorphism $X_\delta$ such that
$(H_\delta,X_\delta,Y_\delta)$ is an $\mathfrak{sl}_2$-triple, i.e. we also have $[H_\delta, X_\delta] = 2X_\delta$ and $[X_\delta, Y_\delta] = H_\delta$.
Now choose an initial metric $h_{x_j}$ on $\EE_{x_j}$ such that the subspaces $\operatorname{Gr}_\delta(\EE_{x_j})$ are orthogonal and such that $H_\delta$ is self-adjoint and $X_\delta$ is adjoint with $Y_\delta$ with respect to $h_{x_j}$. 

Given a trivialization of $\EE\to\oX$ in a coordinate chart $(U,z)$ centered at $x_j$ with a projection $\pi\colon U\to\{x_j\}$, we can pullback by $\pi$ to extend the weight filtration to this chart.

\begin{definition}\label{defn:modelmetric}
    The \textbf{model metric} on $\EE|_{U\setminus\{x_j\}}$ is defined as
    \[h_{\mathrm{mod}}:=\bigoplus_{r,\delta,\lambda}|z|^{-2\delta}\cdot\left|\ln|z|\right|^r\cdot(\pi^*h_{x_j})|_{\operatorname{Gr}_r\operatorname{Gr}_\delta^\lambda(\EE_{x_j})}.\]
\end{definition}

Now we can extend it to a global metric $h_{\mathrm{mod}}$ on $\EE|_X$. Furthermore, if $(\EE,\Phi)$ arises from a parabolic $\SO_0(p,q)$-Higgs bundle, we can make $h_{\mathrm{mod}}$ be a metric on this parabolic $\SO_0(p,q)$-Higgs bundle.

Finally, we give the statement of the Hitchin--Kobayashi correspondence proven in \cite[Theorem 5.1 \& Section 6.1]{biquard2020parabolic}. The statement involves the concepts for general parabolic $G$-Higgs bundles, but we only focus on the case when $G=\mathrm{SL}(n,\mathbb{C})$ or $\SO_0(p,q)$. Note that a metric $h$ on a principal $K^\mathbb{C}$-bundle $\mathbb{E}$ gives a reduction from $\mathbb{E}$ to a principal $K$-bundle $\mathbb{E}_h$. 

\begin{fact}\label{fact:harmonicmetricG}
    For any stable parabolic $G$-Higgs bundle $(\mathbb{E},Q_j,\zeta_j,\Phi)$, there exists a harmonic metric $h$, which is a global smooth section of $\mathbb{E}\left[K\backslash K^{\mathbb{C}}\right]$ satisfying that
    \begin{equation}\label{eq:Hitchin}
        F(\nabla^h)-[\Phi,\tau_h(\Phi)]=0,
    \end{equation}
    or equivalently,
    \[\mathrm{D}^h=\nabla^h+\Phi-\tau_h(\Phi)\]
    is a flat $G$-connection on the principal $G$-bundle obtained by extending the structure group of $\mathbb{E}_h$ to $G$ via the embedding $K\hookrightarrow G$, where $\nabla^h$ denotes unique connection compatible with the holomorphic structure of $\mathbb{E}$ and the metric $h$, $F(\nabla^h)$ denotes its curvature and $\tau_h$ is the conjugation on $\Omega^{1,0}(\mathbb{E}[\mathfrak{p}^\mathbb{C}])$ defined by combining the metric $h$ and the standard conjugation on $X$ from $(1,0)$-forms to $(0,1)$-forms and $h$ is quasi-isometric to the model metric $h_{\mathrm{mod}}$. Moreover, such harmonic metric is unique up to an automorphism of $(\mathbb{E},Q_j,\zeta_j,\Phi)$.
\end{fact}

\begin{remark}
    Two metrics $g_1,g_2$ on a vector bundle are called \textbf{quasi-isometric} (or \textbf{mutually bounded}) if there exists a constant $c\geqslant1$ such that $\|v\|_{g_1}/\|v\|_{g_2}\in[1/c,c]$ for any nonzero vector $v$. Any two smooth metrics are quasi-isometric when the base space is compact.
\end{remark}

Note that when $G=\mathrm{SL}(n,\mathbb{C})$ or $\SO_0(p,q)$, the conjugation $\tau_h(\Phi)$ of $\Phi$ is just $-\Phi^{*_{h_\EE}}$, where $*_{h_\EE}$ denotes the adjoint with respect to the metric $h_{\EE}$.

\begin{definition}
    The map $\mathsf{NAH}$ sending a stable parabolic $G$-Higgs bundle to the monodromy representation $\pi_1(X)\to G$ of the flat principal $G$-bundle induced by the harmonic metric (\prettyref{fact:harmonicmetricG}) is called the \textbf{non-Abelian Hodge correspondence}.
\end{definition}

Hence the process from a polystable $\mathrm{SO}_0(p,q)$-Higgs bundle $(\mathbb{E},\Phi)$ equipped with the harmonic metric to a flat principal $\SO_0(p,q)$-bundle follows the steps below:
\begin{itemize}
    \item[(1)] the $\SO(p,\mathbb{C})\times\SO(q,\mathbb{C})$-principal bundle corresponds the orthonormal unit-frame bundle of $\UU$ and $\VV$ with respect to the symmetric bilinear forms $Q_\UU$, $Q_\VV$ respectively;

    \item[(2)] the $\SO(p)\times\SO(q)$-principal bundle obtained from the reduction $h$ corresponds the product of the orthonormal unit-frame bundle of $\UU$ with respect to $Q_\UU,h_{\UU}$ simultaneously and the orthonormal unit-frame bundle of $\VV$ with respect to $Q_\VV,h_{\VV}$ simultaneously; furthermore this gives real subbundles $\UU_{\mathbb{R}}$ and $\VV_{\mathbb{R}}$ of $\UU$ and $\VV$ respectively and $h_{\UU}|_{\UU_{\mathbb{R}}}=Q_{\UU}|_{\UU_{\mathbb{R}}}$, $h_{\VV}|_{\VV_{\mathbb{R}}}=Q_{\VV}|_{\VV_{\mathbb{R}}}$;

    \item[(3)] the $\SO_0(p,q)$-principal bundle corresponds to the orthonormal unit-frame bundle of $\EE_{\mathbb{R}}=\UU_{\mathbb{R}}\oplus\VV_{\mathbb{R}}$ with respect to the indefinite bilinear metric $(h_\UU\oplus(-h_{\VV}))|_{\EE_{\mathbb{R}}}$ whose signature is $(p,q)$. The flat connection here comes from $\nabla^{h_\EE}+\Phi+\Phi^{*_{h_\EE}}$, where $\nabla^{h_\EE}$ denotes the Chern connection of $h_\EE$.
\end{itemize}

\section{\texorpdfstring{$\alpha_1$}{α1}-cyclic parabolic \texorpdfstring{$\SO_0(2,3)$}{SO0(2,3)}-Higgs Bundles}\label{sec:Higgsestimate}

In this article, we will mainly consider the parabolic $\alpha_1$-cyclic parabolic $\SO_0(2,3)$-Higgs bundles over $(\oX,D)$, where $\oX$ is a compact Riemann surface with a finite subset $D$ (maybe empty) on it satisfying that $X:=\oX\setminus D$ has negative Euler characteristic.

\subsection{Definitions and Basic Properties}

\begin{definition}\label{defn:cyclic}
    A parabolic $\SO_0(2,3)$-Higgs bundle $(\UU,\VV,Q_\UU,Q_\VV,\Phi)$ over $(\oX,D)$ is called \textbf{$\alpha_1$-cyclic} if it satisfies the following properties:
    \begin{itemize}
        \item[(1)] $\UU=\LL_{-1}\oplus\LL_1$, $\VV=\LL_{-2}\oplus\LL_0\oplus\LL_2$ for some holomorphic line bundles $\LL_i$, where $\LL_{-i}\cong\LL_i^{\vee}$;

        \item[(2)] $Q_\UU=\begin{pmatrix}
            0&1\\1&0
        \end{pmatrix}$ and $Q_\VV=\begin{pmatrix}
            0&0&-1\\
            0&-1&0\\
            -1&0&0
        \end{pmatrix},$ where all $1$'s are given by the natural pairing;

        \item[(3)] the Higgs field $\Phi$ is of the following form
        \[\Phi=\begin{pmatrix}
            0&0&0&\gamma^{\vee}&0\\
            \tau^\vee&0&0&0&\gamma\\
            0&\beta^\vee&0&0&0\\
            0&0&\beta&0&0\\
            0&0&0&\tau&0
        \end{pmatrix}\]
        with respect to the decomposition $\EE=\bigoplus_{i=-2}^2\LL_i$, where $\tau|_X\colon\LL_1|_X\to\LL_2|_X\otimes\KK_{X}$ is an isomorphism and $\beta\neq0$.
    \end{itemize}
\end{definition}

\begin{remark}
    If $\beta\equiv0$, the target Lie group reduces into $\SO_0(2,2)$.
\end{remark}

We will always use the following graph to denote the $\alpha_1$-cyclic parabolic $\SO_0(2,3)$-Higgs bundle mentioned below.
        % https://q.uiver.app/#q=WzAsNSxbMCwwLCJcXG1hdGhjYWx7TH1fey0yfSJdLFsxLDAsIlxcbWF0aGNhbHtMfV97LTF9Il0sWzIsMCwiXFxtYXRoY2Fse0x9X3swfSJdLFszLDAsIlxcbWF0aGNhbHtMfV97MX0iXSxbNCwwLCJcXG1hdGhjYWx7TH1fezJ9Il0sWzAsMSwiXFxhbHBoYSIsMl0sWzEsMiwiXFxiZXRhIiwyXSxbMiwzLCJcXGJldGEiLDJdLFszLDQsIlxcYWxwaGEiLDJdLFs0LDEsIlxcZ2FtbWEiLDIseyJjdXJ2ZSI6M31dLFszLDAsIlxcZ2FtbWEiLDIseyJjdXJ2ZSI6M31dXQ==
\begin{equation}\label{eq:cyc}
    (\EE,\Phi)=\begin{tikzcd}
	{\mathcal{L}_{-2}} & {\mathcal{L}_{-1}} & {\mathcal{L}_{0}} & {\mathcal{L}_{1}} & {\mathcal{L}_{2}}
	\arrow["\tau^\vee"', from=1-1, to=1-2]
	\arrow["\beta^\vee"', from=1-2, to=1-3]
	\arrow["\beta"', from=1-3, to=1-4]
	\arrow["\tau"', from=1-4, to=1-5]
	\arrow["\gamma"', curve={height=18pt}, from=1-5, to=1-2]
	\arrow["\gamma^\vee"', curve={height=18pt}, from=1-4, to=1-1]
\end{tikzcd}
\end{equation}

The following definition generalizes the ``assumption A'' defined in \cite[Definition 2.1.9]{filip2021uniformization}.

\begin{definition}\label{defn:AssumptionA}
 We say that the $\alpha_1$-cyclic parabolic $\SO_0(2,3)$-Higgs bundle \prettyref{eq:cyc} satisfies \textbf{assumption A} if for every $x_j\in D$,
\begin{itemize}
    \item[(1)] the parabolic weights are required to be $(\zeta^j,\zeta^j,0,-\zeta^j,-\zeta^j)$ for some $\zeta^j\in[0,1/2)$; 
    
    \item[(2)] when $\zeta^j\neq0$, the weighted reverse isotropic flag is either
\[0\subsetneq(\LL_1\oplus\LL_2)_{x_j}\subsetneq(\LL_0\oplus\LL_1\oplus\LL_2)_{x_j}\subsetneq(\EE)_{x_j}\]
or
\[0\subsetneq(\LL_{-2}\oplus\LL_{-1})_{x_j}\subsetneq(\LL_{-2}\oplus\LL_{-1}\oplus\LL_0)_{x_j}\subsetneq(\EE)_{x_j},\]
and when $\zeta^j=0$, the weighted reverse isotropic flag is the trivial flag $0\subsetneq(\EE)_{x_j};$

\item[(3)] $\tau\colon\LL_1\to\LL_2\otimes\KK_{\oX}(D)$ is an isomorphism.
\end{itemize}
In other words, $(\LL_1)_{x_j}$ and $(\LL_2)_{x_j}$ share the same parabolic weight and $\tau$ is a global isomorphism.

We call such a Higgs bundle is of \textbf{non-zero weights} if $\zeta^j\neq0$ for every $x_j\in D$.
\end{definition}

\begin{remark}
    When the surface $X$ is compact, i.e. $D=\emptyset$, the assumption A is satisfied automatically.
\end{remark}

\begin{remark}
    When $\gamma\equiv0$, an $\alpha_1$-cyclic parabolic $\SO_0(2,3)$-Higgs bundle satisfying assumption A coincides with a real variation of Hodge structure (RVHS) whose Hodge numbers are $(1,1,1,1,1)$ satisfying assumption A introduced by Filip in \cite{filip2021uniformization}.
\end{remark}

By the Milnor--Wood inequality for parabolic Higgs bundles, c.f. \cite[Section 8.2]{biquard2020parabolic}, a necessary condition to get the semistability of an $\alpha_1$-cyclic parabolic $\SO_0(2,3)$-Higgs bundle $(\EE,\Phi)$ represented by \prettyref{eq:cyc} is that $|\operatorname{pardeg}(\LL_1)|\leqslant\deg(\KK_\oX(D))$. Moreover, we give the following criterion of (semi-)stability.

\begin{proposition}\label{prop:stable}
 Suppose that the $\alpha_1$-cyclic parabolic $\SO_0(2,3)$-Higgs bundle $(\EE,\Phi)$ satisfies assumption A.
\begin{itemize}
        \item[(1)] $(\EE,\Phi)$ is semistable if and only if $\gamma\neq0$ or $\operatorname{pardeg}(\LL_1)\leqslant\deg(\KK_{\oX}(D))/2$ when $\gamma\equiv0$. 
        \item[(2)] $(\EE,\Phi)$ is stable if and only if $\gamma\neq0$ or $\operatorname{pardeg}(\LL_1)<\deg(\KK_{\oX}(D))/2$ when $\gamma\equiv0$.
    \end{itemize}
\end{proposition}

\begin{proof}
    We first consider that $\gamma\equiv0$. Then the only proper isotropic $\Phi$-invariant subbundles are $\LL_2$ and $\LL_1\oplus\LL_2$. Hence the semistability condition \prettyref{defn:stable} is equivalent to $\operatorname{pardeg}(\LL_2)\leqslant0$ and $\operatorname{pardeg}(\LL_1)+\operatorname{pardeg}(\LL_2)\leqslant0$. Note that $\LL_1\cong\LL_2\otimes\KK_{\oX}(D)$ and $\LL_1,\LL_2$ share the same parabolic weight. We obtain that $\operatorname{pardeg}(\LL_1)=\operatorname{pardeg}(\LL_2)+\deg(\KK_\oX(D))$. Hence the Higgs bundle is stable if and only if $\operatorname{pardeg}(\LL_1)\leqslant\deg(\KK_\oX(D))/2$. Similarly, the stability holds if and only if $\operatorname{pardeg}(\LL_1)<\deg(\KK_\oX(D))/2$.

     When $\gamma\neq0$, there are no nonzero $\Phi$-isotropic subbundles. To obtain that, we fix a local coordinate $(U,z)$ on $X$ and choose local holomorphic frames $e_i$ on $U$ for each $\LL_i$ such that $e_{-i}=e_i^\vee$ and $\tau(e_1)=e_{2}\dd z$. Set $\beta(e_0)=be_1\dd z,\gamma(e_2)=ce_{-1}\dd z$. Then locally the characteristic polynomial of $\Phi/\dd z$ is $\det(\lambda-\Phi/\dd z)=\lambda(\lambda^4+2b^2c)$. Generically $b^2c\ne0$ and the eigenvectors of $\Phi/\dd z$ are
     \[\begin{cases}
         -c e_{-2}+a e_2&\mbox{for }\lambda=0\\
         b^2c e_{-2}+\omega^3e_{-1}+\omega^2be_0+\omega b^2e_1+b^2e_2&\mbox{for }\lambda=\omega,\mbox{ where }\omega^4=-2b^2c
     \end{cases}.\]
     Now one can readily check these eigenvectors could not be isotropic if $b,c\neq0$. Hence any possible $\Phi$-invariant subbundle is not isotropic if $\beta,\gamma\neq0$.
\end{proof}

\begin{remark}
    A direct degree computation shows that $\beta\neq0$ implies that $\operatorname{pardeg}(\LL_1)\geqslant
    -\deg(\KK_\oX(D))$ and $\gamma\neq0$ implies that $\operatorname{pardeg}(\LL_1)\leqslant\deg(\KK_\oX(D))$. Hence \prettyref{prop:stable} coincides with the bound given by Milnor--Wood inequality.
\end{remark}

The proof of the following lemma for Hitchin section can be found in \cite[Corollary 2.11]{collier2017asymptotics} and there is no difference for our case because the key point is the uniqueness of the harmonic metric and the compatibility between the harmonic metric and the holomorphic bilinear form. We omit its proof.

\begin{lemma}\label{lemma:metricsplit}
    Suppose that the $\alpha_1$-cyclic parabolic $\SO_0(2,3)$-Higgs bundle \prettyref{eq:cyc} is stable. Then the harmonic metric $h_{\EE}$ of a stable  splits as $\bigoplus_{i=-2}^2h_{i}$, where $h_i$ is an Hermitian metric on $\LL_i$. Furthermore, $h_{-i}=h_i^{\vee}$.
\end{lemma}

For the real structure given by the harmonic metric of a stable $\alpha_1$-cyclic parabolic $\SO_0(2,3)$-Higgs bundle, we have the following lemma.

\begin{lemma}\label{lemma:realstructure}
    Suppose that the $\alpha_1$-cyclic parabolic $\SO_0(2,3)$-Higgs bundle \prettyref{eq:cyc} is stable. Then there are real subbundles $(\LL_i)_{\mathbb{R}}\subset\LL_i^{\vee}\oplus\LL_i$ for $i=1,2$ and $(\LL_0)_{\mathbb{R}}\subset\LL_0$ such that the real subbundle given by the harmonic metric of a stable $\alpha_1$-cyclic parabolic $\SO_0(2,3)$-Higgs bundle is $\bigoplus_{i=0}^2(\LL_i)^\mathbb{R}$. Furthermore, $\LL_i$ and $\LL_{-i}$ are conjugate with respect to $(\LL_i)_{\mathbb{R}}$ for $i=1,2$.
\end{lemma}

\begin{proof}
    Below we use $\bar\bullet$ to denote the conjugation of $\bullet$ induced by the real subbundle given by the harmonic metric. Note that with respect to the bilinear form $Q=Q_\UU\oplus Q_{\VV}$, 
    \[Q(\LL_i,\LL_{-j})=0\mbox{ for any }i\neq j.\]
    Therefore, the conjugate of $\LL_i$ must be $\LL_{-i}$ by $Q(\bullet,\bullet)=h(\overline{\bullet},\bullet)$ and \prettyref{lemma:metricsplit}.
\end{proof}

\subsection{Higgs Field Estimates}\label{sec:proofboundedness}

In this section, we will prove the following proposition, which is the key point in the proof of our main results. 

\begin{proposition}\label{prop:boundedness}
    Suppose that the $\alpha_1$-cyclic parabolic $\SO_0(2,3)$-Higgs bundle
    \[(\EE,\Phi)=\begin{tikzcd}
	{\mathcal{L}_{-2}} & {\mathcal{L}_{-1}} & {\mathcal{L}_{0}} & {\mathcal{L}_{1}} & {\mathcal{L}_{2}}
	\arrow["\tau^\vee"', from=1-1, to=1-2]
	\arrow["\beta^\vee"', from=1-2, to=1-3]
	\arrow["\beta"', from=1-3, to=1-4]
	\arrow["\tau"', from=1-4, to=1-5]
	\arrow["\gamma"', curve={height=18pt}, from=1-5, to=1-2]
	\arrow["\gamma^\vee"', curve={height=18pt}, from=1-4, to=1-1]
\end{tikzcd}\] is stable, satisfying assumption A, of non-zero weights and $\operatorname{pardeg}(\LL_1)<\deg(\KK_\oX(D))$. Then there exists a constant $C>0$, such that $\|\tau\|-\|\gamma\|>C$, where $\tau$, $\gamma$ are viewed as sections of $\LL_1^\vee\otimes\LL_2\otimes\KK_{X}$
    and $\LL_2^\vee\otimes\LL_{-1}\otimes\KK_{X}$ respectively, whose metrics are induced by the harmonic metric and the unique complete conformal hyperbolic metric $g_{\mathrm{hyp}}$.
\end{proposition}

Its proof relies on the following Cheng--Yau maximum principle, c.f. \cite[Theorem 8]{cheng1975differential}. Let $\Delta_g$ denote the metric Laplacian with respect to a metric $g$.

\begin{fact}\label{fact:CYmaximum}
    Suppose $(M,h)$ is a complete manifold with Ricci curvature bounded from below. Let $u$ be a $C^2$-function defined on $M$ such that $\Delta_hu\geqslant f(u)$, where $f\colon\mathbb{R}\to\mathbb{R}$ is a function. Suppose there is a continuous positive function $g\colon [a,\infty)\to\mathbb{R}_+$ such that
\begin{itemize}
    \item[(1)] $g$ is non-decreasing;

    \item[(2)] $\liminf_{t\to+\infty}\frac{f(t)}{g(t)}>0$;

    \item[(3)] $\int_a^\infty(\int_b^t g(\tau)\dd \tau)^{-1/2}\dd t<\infty$, for some $b\geqslant a$,
\end{itemize}
 then the function $u$ is bounded from above. Moreover, if $f$ is lower semi-continuous, $f(\sup u)\leqslant 0$.
\end{fact}

In particular, for $\delta> 1$ and a positive constant $c_0$, one can check if $f(t)\geqslant c_0t^\delta$ for $t$ large enough, $g(t) = t^{(\delta+1)/2}$ satisfies the above three conditions.

In this subsection, we always assume that our $\alpha_1$-cyclic parabolic $\SO_0(2,3)$-Higgs bundle satisfies the conditions in \prettyref{prop:boundedness}, i.e. stable, satisfying assumption A, of non-zero weights and $\operatorname{pardeg}(\LL_1)<\deg(\KK_{\oX}(D))$. To make use of \prettyref{fact:CYmaximum}, we need to choose a suitable background metric. 

Due to the stability, the harmonic metric splits as $\oplus_{i=-2}^2h_i$ by \prettyref{lemma:metricsplit}. Since $\tau|_X\colon\LL_1|_X\to\LL_2|_X\otimes\KK_X$ is an isomorphism, we can obtain an Hermitian metric $h_1^\vee\otimes h_2|_X$ on $\KK_X^\vee\cong\mathcal{T}_X$, i.e. an Hermitian metric on the holomorphic tangent bundle and this induces a Riemannian metric $g_h$ on $X$. To use the maximum principle with respect to the background metric $g_h$, we would like to prove its completeness and that its curvature must be bounded from below. 

\begin{remark}
    \prettyref{prop:boundedness} is already known for compact surface $X$ when $\deg(\LL_1)\in[-\deg(\KK_X),0]$, c.f. \cite[Section 4.3]{collier2019geometry}. But we need extra analysis on the behavior of the background metric and the Higgs fields to use the maximum principle when the surface is non-compact.
\end{remark}

Below we use the notation $A\lesssim B$ or $B\gtrsim A$ to denote that there exists some positive constant $c>0$ such that $A\leqslant c\cdot B$.

\begin{lemma}\label{lemma:completeness}
    The metric $g_h$ is quasi-isometric to the unique complete conformal hyperbolic metric $g_{\mathrm{hyp}}$. In particular, it is complete. Moreover, $\|\tau\|$ is bounded by positive constants from above and below.
\end{lemma}

\begin{proof}
    Suppose that $(U_j,z_j)$ are coordinate charts centered at punctures $x_j\in D$ respectively and $\oX\setminus\bigcup_{j=1}^{\deg(D)} U_j$ is compact. $\tau$ can be represented as $v_j^\vee\otimes w_j\otimes\dfrac{\dd z_j}{z_j}$ for some non-vanishing local section $v_j\in\mathrm{H}^0(U_j,\LL_1), w_j\in\mathrm{H}^0(U_j,\LL_2)$, where $v_j^\vee$ (resp. $w_j^\vee$) is the dual of $v_j$ (resp. $w_j$), since $\tau$ is an isomorphism between $\LL_1$ and $\LL_2\otimes\KK_{\oX}(D)$. By \prettyref{fact:harmonicmetricG}, it suffices to prove that 
    \[\dfrac{\|w_j\|_{h_{\mathrm{mod}}}}{\|v_j\|_{h_{\mathrm{mod}}}}\cdot\left\|\dfrac{\dd z_j}{z_j}\right\|_{g_{\mathrm{hyp}}^{\vee}}\in\left[\dfrac{1}{c},c\right]\mbox{ for some positive }c.\]

    Below we recall the definition of the model metric \prettyref{defn:modelmetric}. Since the Higgs bundle is $\alpha_1$-cyclic, $\LL_1$ and $\LL_2$ share the same weight $\delta^j\neq0$ (which equals to $\zeta^j$ or $-\zeta^j$) at $x_j$. Therefore, the weight graded pieces are $\operatorname{Gr}_{-\delta^j}(\EE_{x_j})=(\LL_{-2}\oplus\LL_{-1})_{x_j}$, $\operatorname{Gr}_{0}(\EE_{x_j})=(\LL_0)_{x_j}$ and $\operatorname{Gr}_{\delta^j}(\EE_{x_j})=(\LL_{1}\oplus\LL_{2})_{x_j}$. With respect to the basis $w_j^\vee,v_j^\vee,1,v_j,w_j$, the residue $\operatorname{Res}_{x_j}\Phi$ can be represented by
    \[\begin{pmatrix}
            0&0&0&\operatorname{Res}_{x_j}\gamma^{\vee}&0\\
            1&0&0&0&\operatorname{Res}_{x_j}\gamma\\
            0&\operatorname{Res}_{x_j}\beta^\vee&0&0&0\\
            0&0&\operatorname{Res}_{x_j}\beta&0&0\\
            0&0&0&1&0
        \end{pmatrix},\]
        and the graded residue $\operatorname{GrRes}_{x_j}\Phi$ is given by
        \[\begin{pmatrix}
            0&0&0&0&0\\
            1&0&0&0&0\\
            0&0&0&0&0\\
            0&0&0&0&0\\
            0&0&0&1&0
        \end{pmatrix}.\]
        Note that it has no semisimple part and the only eigenvalue is $0$. Therefore, the eigenvalue decomposition $\operatorname{Gr}_\delta(\EE_{x_j})=\operatorname{Gr}_\delta^0(\EE_{x_j})$ is trivial for any $\delta\in\mathbb{R}$. The weight filtration is given by
        \[W_{-1}\operatorname{Gr}_{\delta^j}(\EE_{x_j})=W_{0}\operatorname{Gr}_{\delta^j}(\EE_{x_j})=(\LL_2)_{x_j}, W_{1}\operatorname{Gr}_{\delta^j}(\EE_{x_j})=(\LL_{1}\oplus\LL_{2})_{x_j},\]
        \[W_{-1}\operatorname{Gr}_{-\delta^j}(\EE_{x_j})=W_{0}\operatorname{Gr}_{-\delta^j}(\EE_{x_j})=(\LL_{-1})_{x_j}, W_{1}\operatorname{Gr}_{-\delta^j}(\EE_{x_j})=(\LL_{-2}\oplus\LL_{-1})_{x_j},\]
        \[W_0\operatorname{Gr}_{0}(\EE_{x_j})=(\LL_{0})_{x_j}.\]
        Hence the model metric is
        \[\operatorname{diag}\left(|z_j|^{2\delta^j}\cdot|\ln|z_j||,|z_j|^{2\delta^j}\cdot|\ln|z_j||^{-1},1,|z_j|^{-2\delta^j}\cdot|\ln|z_j||,|z_j|^{-2\delta^j}\cdot|\ln|z_j||^{-1}\right).\]
    This implies that \[\dfrac{\|w_j\|_{h_{\mathrm{mod}}}}{\|v_j\|_{h_{\mathrm{mod}}}}=\sqrt{\dfrac{|z_j|^{-2\delta^j}\cdot|\ln|z_j||^{-1}}{|z_j|^{-2\delta^j}\cdot|\ln|z_j||}}=\dfrac{1}{\ln|z_j|}\]
    over $(U_j,z_j)$. 
    
    Due to $g_{\mathrm{hyp}}$ is the complete hyperbolic metric compatible with the complex structure, $x_j$ is a cusp of $g_{\mathrm{hyp}}$, hence we get that $\|\mathrm{d}z_j/z_j\|_{g_{\mathrm{hyp}}^{\vee}}=1/\|z_j\cdot\partial/\partial z_j\|_{g_{\mathrm{hyp}}}=|\ln|z_j||$ by \prettyref{example:cusphyperbolic}. Therefore, \[\dfrac{\|w_j\|_{h_{\mathrm{mod}}}}{\|v_j\|_{h_{\mathrm{mod}}}}\cdot\left\|\dfrac{\dd z_j}{z_j}\right\|_{g_{\mathrm{hyp}}^{\vee}}=\dfrac{1}{|\ln|z_j||}\cdot|\ln|z_j||=1.\]
\end{proof}

Below we do some local analysis. Fix a local coordinate $(U,z=x+\iu y)$ on $X$. By choosing holomorphic frames $e_i$ for each $\LL_i$ such that $e_{-i}=e_i^\vee$ and $\tau(e_1)=e_2\dd z$, the harmonic metric $h_\EE$ can be written as $\sum_{i=-2}^2H_i\overline{e_i^\vee}\otimes e_i^\vee$ for positive smooth functions $H_i$ by \prettyref{lemma:metricsplit}. Furthermore, we have $H_{-i}=H_i^{-1}$ and $H_0=1$. Let $\Delta$ denote the coordinate Laplacian, i.e. $\partial^2/\partial z\bar\partial z=(\partial^2/\partial x^2+\partial^2/\partial y^2)/4$. Note that 
\begin{equation}\label{eq:conformal}
    g_h(\partial/\partial x,\partial/\partial x)=g_h(\partial/\partial y,\partial/\partial y)=2(h_1^{\vee}\otimes h_2)(\partial/\partial z,\partial/\partial z)=2H_{-2}^{-1}H_{-1}.
\end{equation}
We use $|\bullet|^2$ to denote the coordinate norm for a field with respect to the local coordinate and frames $e_i$, $\dd z$. With our choice of frame, we have $|\tau|^2=1$. 

Now the Hitchin's self-dual equation \prettyref{eq:Hitchin} implies that:
\begin{equation}\label{eq:Hit}
    \begin{cases}
        \Delta\ln H_{-2}=H_{-2}H_{1}^{-1}|\gamma^\vee|^2-H_{-2}^{-1}H_{-1},\\
        \Delta\ln H_{-1}=H_{-1}H_{-2}^{-1}+H_{-1}H_{2}^{-1}|\gamma|^2-H_{-1}^{-1}|\beta|^2
    \end{cases}
\end{equation}
by taking projection onto $\LL_{-2}$ and $\LL_{-1}$.

\begin{lemma}\label{lemma:curvaturebound}
    
    The curvature $K_{g_h}$ of $g_h$ is nowhere smaller than $-2$.
\end{lemma}

\begin{proof}
By \prettyref{eq:conformal}, it follows from that
    \[\begin{aligned}
        K_{g_h}=&-2\cdot\dfrac{\Delta\ln (2H_{-2}^{-1}H_{-1})}{2H_{-2}^{-1}H_{-1}}\quad(\mbox{by curvature fomula of conformal metric})\\
        =&-\dfrac{-\Delta\ln H_{-2}+\Delta\ln H_{-1}}{H_{-2}^{-1}H_{-1}}\\
        =&-\dfrac{2H_{-1}H_{-2}^{-1}-H_{-1}^{-1}|\beta|^2}{H_{-2}^{-1}H_{-1}}\quad(\mbox{by \prettyref{eq:Hit} and }|\gamma|^2=|\gamma^\vee|^2)\\
        \geqslant&-2.
    \end{aligned}\]
\end{proof}

\begin{proof}[Proof of \prettyref{prop:boundedness}]
    We first show that $\|\gamma\|\to 0$ when tending to a puncture $x\in D$. Let $\delta$ be the parabolic weight of $\LL_2$, then $\delta\neq 0$. If $\delta>0$, with respect to a compatible basis at $x$, we have $\gamma=O(1)\dd z/z$. Then by \prettyref{fact:harmonicmetricG} we obtain that $\|\gamma\|\lesssim|z|^{2\delta}\cdot|\ln|z||$ and when $z\to 0$, $\|\gamma\|\to 0$. Otherwise $\delta<0$, with respect to a compatible basis at $x$, we have $\gamma=O(z)\dd z/z$. Then by \prettyref{fact:harmonicmetricG} we obtain that $\|\gamma\|\lesssim|z|^{1+2\delta}\cdot|\ln|z||$ and when $z\to 0$, $\|\gamma\|\to 0$.

    Set $u=\|\gamma\|^2/\|\tau\|^2$. Locally we have 
    \[u=\dfrac{H_{-2}H_1^{-1}|\gamma|^2}{H_{-1}H_{-2}^{-1}}=H_{-2}^2|\gamma|^2.\]
    Since $\gamma$ is a holomorphic section, we obtain that
    \[\begin{aligned}
        \Delta\ln u=&2\Delta\ln H_{-2}\\
        =&2(H_{-2}H_1^{-1}|\gamma|^2-H_{-2}^{-1}H_{-1})\\
        =&2(H_{-2}^{-1}H_{-1})(u-1).
    \end{aligned}\]
    Hence globally we have that $\Delta_{g_h}\ln u=u-1$ and this implies that $\Delta_{g_h}u\geqslant u(u-1)$. Note that by \prettyref{lemma:completeness} and \prettyref{lemma:curvaturebound} the right-hand side has a quadratic growth rate. Hence the background metric $g_h$ satisfies the conditions of \prettyref{fact:CYmaximum} and we then obtain that $\sup u\leqslant1$. 
    
    Moreover, if the supremum is attained at some point $x\in X$, then $u\equiv 1$ by the strong maximum principle. When $D\neq\emptyset$, this contradicts with $\|\gamma\|\to0$ when tending to $D$. When $D=\emptyset$, i.e. $X=\oX$ is compact, This implies that $\gamma\colon\LL_2\to\LL_{-1}\otimes\KK_X$ is also an isomorphism. Therefore, \[\deg(\LL_1)=\deg(\LL_2)+\deg(\KK_X)=\deg(\LL_{-1})+2\deg(\KK_X).\]
    Hence $\deg(\LL_1)=\deg(\KK_{X})$, contradiction. Therefore, $\|\tau\|>\|\gamma\|$ over $X$. And when tending to a puncture, $\|\tau\|-\|\gamma\|$ has a positive lower bound by $\|\gamma\|\to0$ and \prettyref{lemma:completeness}. Hence $\|\tau\|-\|\gamma\|$ has a positive lower bound over $X$.
\end{proof}

\begin{remark}\label{rem:max}
    When $X$ is compact and $\deg(\LL_1)=\deg(\KK_X)$, $\gamma\colon\LL_2\to\LL_{-1}\otimes\KK_X$ is also an isomorphism and $\|\tau\|\equiv\|\gamma\|$. Such Higgs bundles achieve the maximal Toledo invariant and there would be strong restrictions on the underlying bundle. More precisely, $\LL_1\cong\LL_2\otimes\KK_X\cong\LL_{-1}\otimes\KK_X^2$ shows that $\LL_1\cong\II\otimes\KK_X$ where $\II$ is a line bundle with $\II^2\cong\mathcal{O}_X$. Now the $\SO_0(2,3)$-Higgs bundle can be represented by
    % https://q.uiver.app/#q=WzAsNSxbMSwxLCJcXElJIl0sWzAsMSwiXFxJSVxcb3RpbWVzXFxLS19YIl0sWzIsMSwiXFxJSVxcb3RpbWVzXFxLS19YXlxcdmVlIl0sWzEsMiwiXFxtYXRoY2Fse099X1giXSxbMSwwLCJcXElJIl0sWzEsMCwiXFxtYXRoYmZ7MX0iXSxbMCwyLCJcXG1hdGhiZnsxfSJdLFszLDEsIlxcYmV0YSIsMCx7ImN1cnZlIjotMn1dLFsyLDMsIlxcYmV0YV5cXHZlZSIsMCx7ImN1cnZlIjotMn1dLFsxLDQsIlxcbWF0aGJmezF9IiwwLHsiY3VydmUiOi0yfV0sWzQsMiwiXFxtYXRoYmZ7MX0iLDAseyJjdXJ2ZSI6LTJ9XV0=
\[\begin{tikzcd}
	& \II \\
	{\II\otimes\KK_X} & \II & {\II\otimes\KK_X^\vee} \\
	& {\mathcal{O}_X}
	\arrow["{\mathbf{1}}", curve={height=-12pt}, from=1-2, to=2-3]
	\arrow["{\mathbf{1}}", curve={height=-12pt}, from=2-1, to=1-2]
	\arrow["{\mathbf{1}}", from=2-1, to=2-2]
	\arrow["{\mathbf{1}}", from=2-2, to=2-3]
	\arrow["{\beta^\vee}", curve={height=-12pt}, from=2-3, to=3-2]
	\arrow["\beta", curve={height=-12pt}, from=3-2, to=2-1]
\end{tikzcd}.\]
By changing the basis of the middle two $\II$, it is equivalent to
% https://q.uiver.app/#q=WzAsNSxbMSwxLCJcXElJIl0sWzAsMSwiXFxJSVxcb3RpbWVzXFxLS19YIl0sWzIsMSwiXFxJSVxcb3RpbWVzXFxLS19YXlxcdmVlIl0sWzEsMiwiXFxtYXRoY2Fse099X1giXSxbMSwwLCJcXElJIl0sWzEsMCwiXFxtYXRoYmZ7MX0iXSxbMCwyLCJcXG1hdGhiZnsxfSJdLFszLDEsIlxcYmV0YSIsMCx7ImN1cnZlIjotMn1dLFsyLDMsIlxcYmV0YV5cXHZlZSIsMCx7ImN1cnZlIjotMn1dLFs0LDAsIlxcb3BsdXMiLDEseyJzdHlsZSI6eyJib2R5Ijp7Im5hbWUiOiJub25lIn0sImhlYWQiOnsibmFtZSI6Im5vbmUifX19XV0=
\[\begin{tikzcd}
	& \II \\
	{\II\otimes\KK_X} & \II & {\II\otimes\KK_X^\vee} \\
	& {\mathcal{O}_X}
	\arrow["\oplus"{description}, draw=none, from=1-2, to=2-2]
	\arrow["{\mathbf{1}}", from=2-1, to=2-2]
	\arrow["{\mathbf{1}}", from=2-2, to=2-3]
	\arrow["{\beta^\vee}", curve={height=-12pt}, from=2-3, to=3-2]
	\arrow["\beta", curve={height=-12pt}, from=3-2, to=2-1]
\end{tikzcd}.\]
Hence the corresponding representation $\rho$ factors through $\mathrm{O}(2,2)\times\mathrm{O}(1)$. Such Higgs bundle is polystable and not stable as $\mathrm{SL}(5,\mathbb{C})$-Higgs bundle. However, it is stable as an $\SO_0(2,3)$-Higgs bundle by \prettyref{prop:stable}. 

When $\II=\mathcal{O}_{X}$, the corresponding representation factors through $\SO_{0}(2,2)$ and the Higgs bundle is equivalent to
% https://q.uiver.app/#q=WzAsMixbMSwwLCJcXElJXFxvdGltZXNcXEtLX1heey0xLzJ9Il0sWzAsMCwiXFxJSVxcb3RpbWVzXFxLS19YXnsxLzJ9Il0sWzEsMCwiXFxtYXRoYmZ7MX0iLDAseyJjdXJ2ZSI6LTJ9XSxbMCwxLCJxXzIiLDAseyJjdXJ2ZSI6LTJ9XV0=
\[\mathcal{O}_X\oplus\left(\left(\begin{tikzcd}
	{\KK_X^{1/2}} & {\KK_X^{-1/2}}
	\arrow["{\mathbf{1}}", curve={height=-18pt}, from=1-1, to=1-2]
	\arrow["{q_2}", curve={height=-18pt}, from=1-2, to=1-1]
\end{tikzcd}\right)\otimes\left(\begin{tikzcd}
	{\KK_X^{1/2}} & {\KK_X^{-1/2}}
	\arrow["{\mathbf{1}}", curve={height=-18pt}, from=1-1, to=1-2]
	\arrow["{q_2}", curve={height=-18pt}, from=1-2, to=1-1]
\end{tikzcd}\right)\right),\]
where $q_2\in\mathrm{H}^0(X,\KK_X^2)$ is a quadratic differential. Therefore, the corresponding representation comes from $\rho'\otimes\rho'$ where $\rho'\colon\pi_1(X)\to\mathrm{SL}(2,\mathbb{R})$ is a Fuchsian representation. So $\rho\colon\pi_1(X)\to\SO_0(2,3)$ is $\{\alpha_1\}$-Anosov but not $\{\alpha_2\}$-Anosov. Moreover, for general $\II$, $\rho$ is also $\{\alpha_1\}$-Anosov by \cite{burger2010higher} since it is a maximal representation but not $\{\alpha_2\}$-Anosov by \cite{davalo2024maximal} since it is not a Hitchin representation.
\end{remark}

\begin{remark}\label{rem:zeroweight}
    When $\zeta^j=0$ at some punctures $x_j\in D$, the only weight graded piece is the total fiber at $x_j$ whose weight is $0$. If we require $\gamma=O(1)\dd z_j$ around $x_j$, then the graded residue of the Higgs field is of the form
    \[\begin{pmatrix}
        0&0&0&0&0\\
        1&0&0&0&0\\
        0&\operatorname{Res}_{x_j}\beta&0&0&0\\
        0&0&\operatorname{Res}_{x_j}\beta&0&0\\
        0&0&0&1&0
    \end{pmatrix},\]
    where $\operatorname{Res}_{x_j}\beta=O(1)$. It gives the same estimates on $\|\tau\|$, $\|\gamma\|$ and $g_h$.
\end{remark}

\section{Index Estimates}\label{sec:estimates}

In this section we give some estimates for a general class of functions. The estimates can be used to establish domination property for Higgs bundles by choosing different functions given by different Higgs bundles. The crucial estimates below can be found in \cite[Section 2.2]{filip2021uniformization} for a special function $f_w$. We rewrite the statements and their proofs for completeness.

We fix a complete Riemannian manifold $(M,g)$ with the distance function $d\colon M\times M\to\mathbb{R}$ and mainly consider the smooth function $f\colon M\to\mathbb{R}\in C^\infty(M;\mathbb{R})$ satisfying part of the following conditions:
\begin{itemize}
    \item[(S1)] $f$ is a non-negative Morse function, i.e. $f\geqslant0$ and all critical points of $f$ are non-degenerate;

    \item[(S2)] $\|\dd f\|_g\gtrsim f$;

    \item[(S3)] $\|\dd f\|_g\gtrsim f^{1/2}$,
\end{itemize}
where $\|-\|_g$ denotes the norm associated with $g$.

\subsection{Auxiliary Results}

In \cite[Theorem 2.2.8]{filip2021uniformization}, the following mountain pass lemma and the Ekeland variational principle for Riemannian manifolds are used in his proof. 

\begin{lemma}\label{lemma:mountainpass}
    Suppose $F\colon M\to\mathbb{R}\in C^2(M;\mathbb{R})$ is a twice continuously differentiable function satisfying that
    \begin{enumerate}
        \item[(1)] $F(x_0)=0$ for some $x_0\in M$.

        \item[(2)] There exists $\alpha>0$ and $r>0$ such that $F(x)\geqslant\alpha$ for any $x$ with $d(x,x_0)=r$.

        \item[(3)] There exists an $x_1\in M$ such that $d(x_0,x_1)>r$ and $F(x_1)<\alpha$.
    \end{enumerate}
    Then there exists $c\geqslant\alpha$ and a sequence $(y_n)_{n=1}^{\infty}\in M^{\mathbb{N}}$ such that $F(y_n)\to c$ and $\|\nabla F(y_n)\|_g\to 0$, where $\nabla F$ denotes the gradient of $F$ with respect to the Riemannian metric $g$.
\end{lemma}

In \cite[Theorem 3.1]{bisgard2015mountain}, J. Bisgard proved \prettyref{lemma:mountainpass} for the standard Euclidean space, we point out that his proof is also effective for any complete Riemannian manifold.

\begin{proof}[Proof of \prettyref{lemma:mountainpass}]
    Set $w(x)=\dfrac{\|\nabla F(x)\|_g}{1+\|\nabla F(x)\|_g^2}$, we consider the normalized gradient flow $\varphi_t\colon M\to M$ which is generated by the vector field $-w(x)\nabla F(x)$. In other words, we have
    \[\begin{cases}
        \dfrac{\dd\varphi_t}{\dd t}\bigg|_{t=0}(x)=-w(x)\nabla F(x),\\
        \varphi_0(x)=x.
    \end{cases}\]
    Note that
    \[
        \|-w(x)\nabla F(x)\|_g=\dfrac{\|\nabla F(x)\|_g^2}{1+\|\nabla F(x)\|_g^2}\in[0,1).
    \]
    We know that the flow $\varphi_t$ exists all the time because $M$ is complete. Since
    \[
    \begin{aligned}
        &\dfrac{\dd}{\dd t}\bigg|_{t=0}F(\varphi_t(x))\\
        =&\dd F(-w(x)\nabla F(x))\\
        =&g(-w(x)\nabla F(x),\nabla F(x))\\
        =&-w(x)\|\nabla F(x)\|_g^2\leqslant 0,
    \end{aligned}
    \]
    $F$ is decreasing along $\varphi_t(x)$.
    We claim that $d(x_0,\varphi_t(x_0))<r$ for all $t>0$. Otherwise, there exists $t_0>0$ such that $d(x_0,\varphi_{t_0}(x_0))=r$, then
    \[\alpha\leqslant F(\varphi_{t_0}(x_0))\leqslant F(x_0)=0,\]
    contradiction. 
    
    Similarly we have $d(x_0,\varphi_t(x_1))>r$ for all $t>0$. Suppose now that there is a path $\gamma\colon[0,1]\to M$ connecting $x_0$ and $x_1$, i.e. $\gamma(0)=x_0$ and $\gamma(1)=x_1$. For any $t\in\mathbb{R}$, we set $\gamma_t:=\varphi_t\circ\gamma$. For any non-negative integer $n\in\mathbb{N}$, since 
    \[d(x_0,\gamma_n(0))=d(x_0,\varphi_n(x_0))<r<d(x_0,\varphi_n(x_1))=d(x_0,\gamma_n(1)),\]
    there exists $s_n\in(0,1)$ such that $d(x_0,\gamma_n(s_n))=r$. Hence
    \[\alpha\leqslant F(\gamma_n(s_n))\leqslant\max_{s\in[0,1]}F(\gamma_n(s))=:F(\gamma_n(s_n^\prime)),\]
    where $F(\gamma_n(\bullet))$ achieves its maximum at $s_n'\in[0,1]$. Suppose that $s^*\in[0,1]$ is an accumulation point of $\{s_n^\prime\}$. We claim that $F(\gamma_n(s^*))\geqslant\alpha$ for all $n\in\mathbb{N}$. Otherwise there exists $N\in\mathbb{N}$ such that $F(\gamma_{N}(s^*))<\alpha$. Therefore there exists a subsequence $(s_{n_j}^\prime)$ of $(s_n^\prime)$ converging to $s^*$ and $J\in\mathbb{N}$ such that $F(\gamma_{N}(s_{n_j}^\prime))<\alpha$ for any $j>J$. Thus we can take $j'>J$ large enough such that $n_{j'}>N$, and then
    \[\alpha\leqslant F(\gamma_{n_{j'}}(s_{n_{j'}}^\prime))\leqslant F(\gamma_{N}(s_{n_{j'}}^\prime))<\alpha,\]
    contradiction. 

    Now since the sequence $F(\gamma_n(s^*))$ is decreasing and bounded from below by $\alpha$, $c:=\lim_{n\to\infty}F(\gamma_n(s^*))$ exists and $c\geqslant\alpha$. We know the integral
    \[\int_{0}^{+\infty}w(\gamma_t(s^*))\|\nabla F(\gamma_t(s^*))\|_g^2\dd t=\int_{0}^{+\infty}-\dfrac{\dd}{\dd t}F(\gamma_t(s^*))\dd t=F(\gamma(s^*))-c\]
    is finite. Hence there is a sequence $(t_n)\in\mathbb{R}_{>0}^{\mathbb{N}}$ such that
    \[\lim_{n\to+\infty}w(\gamma_{t_n}(s^*))\|\nabla F(\gamma_{t_n}(s^*))\|_g^2=0.\]
    Let $y_n:=\gamma_{t_n}(s^*)$. Then $\lim_{n\to\infty}F(y_n)=c\geqslant\alpha$ and
    \[w(y_n)\|\nabla F(y_n)\|_g^2=\dfrac{\|\nabla F(y_n)\|_g^3}{1+\|\nabla F(y_n)\|_g}\to 0\quad\mbox{when }n\to\infty\]
    implies that $\lim_{n\to\infty}\|\nabla F(y_n)\|_g=0$.
\end{proof}

We will also use the following fact which is implied by the Ekeland variational principle which is proven in \cite[Proposition 2.2]{alias2016maximum}. 

\begin{fact}\label{fact:Ekeland}
 Suppose $u\in C^1(M;\mathbb{R})$ is a continuously differentiable function with $u^*=\sup_{M}u<+\infty$. Then, for every sequence $(y_n)_{n=1}^{\infty}\in M^\mathbb{N}$ such that $u(y_n)\to u^*$ as $n\to\infty$, there exists a sequence $(x_n)_{n=1}^{\infty}\in M^\mathbb{N}$ with the properties
    \begin{itemize}
        \item[(1)] $u(x_n)\to u^*$;

        \item[(2)] $\|\nabla u(x_n)\|_g\to 0$;

        \item[(3)] $d(x_n,y_n)\to0$.
    \end{itemize}
\end{fact}

\subsection{Exponential Growth}

Below we state the index estimates of function $f$ satisfying (S1)-(S3) given in \cite{filip2021uniformization}.

\begin{theorem}\label{thm:index}
Given a smooth function $f\in C^\infty(M;\mathbb{R})$ satisfying (S1) and (S2).
    \begin{itemize}
        \item[(1)] The only critical points of $f$ are local minima, which occur when $f(x) = 0$.

        \item[(2)] $\inf_{M}f= 0$, and furthermore if $f$ attains its infimum (indeed minimum $0$) at some $x_{\mathrm{min}}\in M$, then for any sequence $(x_n)_{n=1}^\infty$ with $\lim_{n\to\infty}f(x_n)=0$, $(x_n)_{n=1}^\infty$ converges to $x_{\mathrm{min}}$. In particular, $f$ has at most one critical point.

        \item[(3)] In addition, if $f$ satisfies (S3) as well, then $f$ has precisely one critical point.
    \end{itemize}
\end{theorem}

\begin{proof}

It follows from (S2) that if $\dd f(x)=0$, then $f(x)\leqslant|\dd f(x)|=0$ and by (S1) we have $f(x)=0$. So the only critical points are local minima. 

We take $u=-f$ in \prettyref{fact:Ekeland} and obtain a sequence $(x_n)_{n=1}^\infty$ which satisfies that $f(x_n)\to\inf_{M}f$ and $\|\dd f(x_n)\|_g=\|\nabla f(x_n)\|_{g}\to0$. Therefore, by (S2) and (S1), we have that $\inf_{M}f_v=0$.

Now suppose $f$ attains its minimum $0$ at some point $x_{\mathrm{min}}\in M$ and there is a sequence $(x_n)_{n=1}^\infty\in M^{\mathbb{N}}$ such that $\lim_{n\to\infty}f(x_n)=0$. For any $r>0$ small enough, there is a positive number $\alpha>0$ such that $f(x)\geqslant\alpha$ for any $x$ with $d(x,x_{\mathrm{min}})=r$ since $x_{\mathrm{min}}$ is an isolated zero by (S1). 

By $\lim_{n\to\infty}f(x_n)=0$, there exists an integer $N>0$ such that for any $n>N$, $f(x_n)<\alpha$. Then $d(x_{\mathrm{min}},x_n)\leqslant r$ for any $n>N$. Otherwise, there is an $x_{n}$ satisfying $f(x_n)<\alpha$ and $d(x_n,x_{\mathrm{min}})>r$. So $f,x_{\mathrm{min}},x_n$ satisfy the conditions in \prettyref{lemma:mountainpass} and we get there exists a sequence $(y_n)_{n=1}^{\infty}$ such that $f(y_n)>c$ for some positive constant $c>0$ and $\|\nabla f(y_n)\|_{g}\to0$. However, by (S2) we obtain that $f(y_n)\to 0$, contradiction. This shows that $(x_n)_{n=1}^\infty$ converges to $x_{\mathrm{min}}$. In particular, this implies that $f$ cannot have two critical points. 

Below we assume that $f$ satisfies (S3) as well and prove $f$ attains its infimum exactly once by using the gradient flow. For any $x\in M$, let $\gamma\colon[0,t_0)\to M$ be the unique curve satisfying that \[\begin{cases}
    \dfrac{\dd\gamma}{\dd t}(t)=-\nabla f(\gamma(t))\\
    \gamma(0)=x
\end{cases}\]
with the maximum existence time $t_0>0$. Set $h:=f\circ\gamma$, then
\[\begin{aligned}
    &\dfrac{\dd h}{\dd t}(t)\\
    =&\dd f(-\nabla f)(\gamma(t))\\
    =&-\|\nabla f(\gamma(t))\|_g^2\leqslant0
\end{aligned}\]
and there exists $\varepsilon_0>0$ such that
\[-\dfrac{\dd h}{\dd t}\geqslant \varepsilon_0\cdot h\]
by (S3). Hence for any $[a,b]\subset[0,t_0)$, by integrating the equation above, we have 
\[h(b)\leqslant h(a)\cdot\exp(-\varepsilon_0\cdot(b-a)).\]
Then
\[\begin{aligned}
    \left(\operatorname{length}(\gamma|_{[a,b]})\right)^2
    =&\left(\int_a^b\left\|\dfrac{\dd \gamma}{\dd t}\right\|_g\dd t\right)^2\\
    \leqslant&\left(\int_a^b\left\|\dfrac{\dd \gamma}{\dd t}\right\|_g^2\dd t\right)\cdot(b-a)\quad(\mbox{by Chauchy--Schwarz inequality})\\
    =&\left(\int_a^b-\dfrac{\dd h}{\dd t}\dd t\right)\cdot(b-a)\\
    =&\left(h(a)-h(b)\right)\cdot(b-a)\\\leqslant& h(a)\cdot(b-a).
\end{aligned}\]
For any $t\in[0,t_0)$, the above inequality implies that
\begin{equation}\label{eq:distestimate}
\begin{aligned}
    d(\gamma(0),\gamma(t))
    \leqslant&\operatorname{length}(\gamma|_{[0,t]})\\
    =&\sum_{n=1}^{\lfloor t\rfloor-1}\operatorname{length}(\gamma|_{[n,n+1]})+\operatorname{length}(\gamma|_{[\lfloor t\rfloor,t]})\\
    \leqslant&\sum_{n=1}^{\lfloor t\rfloor}h(n)^{1/2}\\
    \leqslant&\sum_{n=1}^{\infty}h(n)^{1/2}\\
    \leqslant&h(0)^{1/2}\cdot\sum_{n=1}^{\infty}\exp(-\varepsilon_0\cdot n/2)\\
    =&h(0)^{1/2}\cdot\dfrac{\exp(-\varepsilon_0/2)}{1-\exp(-\varepsilon_0/2)}.
\end{aligned}    
\end{equation}
Therefore, the image of $\gamma$ stays in a compact subset and this implies that not only $t_0=+\infty$, but also $\gamma(+\infty)=\lim_{t\to+\infty}\gamma(t)$ exists, which is the required minima of $f$.
\end{proof}

As a corollary, we have the following exponential growth lemma of $f$.

\begin{lemma}\label{lemma:exponentialgrowth}
    Suppose that $f$ satisfies (S1)-(S3) and achieves its minimum (which is necessarily exists and unique by \prettyref{thm:index}) at $x_{\mathrm{min}}$.
    Then there exist constants $C_1,C_2,\varepsilon > 0$, such that
\[f(x)\geqslant C_1\cdot\exp(\varepsilon\cdot d(x_{\mathrm{min}},x))-C_2,\forall x\in M,\]
    where $d$ denotes the distance function of $(M,g)$.
\end{lemma}

\begin{proof}
    It suffices to prove that there exist constants $c_1,c_2>0$ such that
    \[d(x_{\mathrm{min}},x)\leqslant c_1+c_2\cdot\ln(1+f(x)).\]

    Now fix an arbitrary $x\in M$. We make use of the integral curve $\gamma$ of the gradient flow again, i.e. 
   \[\begin{cases}
    \dfrac{\dd\gamma}{\dd t}(t)=-\nabla f(\gamma(t))\\
    \gamma(0)=x
\end{cases}\]
and we know that the maximal existence time of $\gamma$ is $+\infty$ with $\gamma(+\infty)=x_{\mathrm{min}}$ by \prettyref{thm:index} and we set $h=f\circ\gamma$ again. Then
\[-\dfrac{\dd h}{\dd t}(t)=\|\nabla f(\gamma(t))\|_{g}^2\geqslant\varepsilon_1\cdot (h(t))^2\]
for some constant $\varepsilon_1>0$ by (S2). Hence 
\[\dfrac{\dd}{\dd t}\left(\dfrac{1}{h(t)}\right)\geqslant\varepsilon_1\]
and by integrating we obtain that for any $t>0$, 
\[\dfrac{1}{h(t)}-\dfrac{1}{h(0)}\geqslant\varepsilon_1\cdot t.\]

If $h(0)<2$, then by \prettyref{eq:distestimate} we obtain that
\begin{equation}\label{eq:smalldistestimate}
    d(x,x_{\mathrm{min}})\leqslant\sqrt{2}\cdot\dfrac{\exp(-\varepsilon_0/2)}{1-\exp(-\varepsilon_0/2)}=:c_1,
\end{equation}
where $\varepsilon_0>0$ is the constant such that $\|\nabla f\|_{g}^2\geqslant\varepsilon_0 \cdot f$ which is provided by (S3).

Since $h(+\infty)=0$, there exists $t_1>0$ such that $h(t_1)=h(0)/2$. Thus $t_1\leqslant1/(\varepsilon_1\cdot h(0))$. Moreover, we obtain that
\[\begin{aligned}
    d(x,\gamma(t_1))
    \leqslant&\operatorname{length}(\gamma|_{[0,t_1]})\\
    =&\int_{0}^{t_1}\left\|\dfrac{\dd\gamma}{\dd t}\right\|_{g}\dd t\\
    =&\int_{0}^{t_1}\left\|\nabla f(\gamma(t))\right\|_{g}\dd t\\
    \leqslant&\sqrt{\left(\int_{0}^{t_1}\dfrac{\dd t}{t+1}\right)\cdot\left(\int_{0}^{t_1}(t+1)\cdot\left\|\nabla f(\gamma(t))\right\|_{g}^2\dd t\right)}\quad(\mbox{by Cauchy--Schwarz inequality})\\
    \leqslant&\sqrt{\ln(t_1+1)\cdot(t_1+1)\cdot\left(\int_{0}^{t_1}\left\|\nabla f(\gamma(t))\right\|_{g}^2\dd t\right)}\\
    =&\sqrt{\ln(t_1+1)\cdot(t_1+1)\cdot\left(\int_{0}^{t_1}-\dfrac{\dd h}{\dd t}(t)\dd t\right)}\\
    \leqslant& \sqrt{\ln(t_1+1)\cdot(t_1+1)\cdot\dfrac{h(0)}{2}}\\
    \leqslant& \sqrt{t_1\cdot(t_1+1)\cdot\dfrac{h(0)}{2}}\\
    \leqslant&\sqrt{\dfrac{t_1+1}{2\varepsilon_1}}. 
\end{aligned}\]

Now suppose that $h(0)=f(x)\geqslant2$. We know that $t_1\leqslant{1}/{(2\cdot\varepsilon_1)}$. Therefore,
\[d(x,\gamma(t_1))\leqslant\sqrt{\dfrac{2\varepsilon_1+1}{(2\varepsilon_1)^2}}=:c_3\]
Let $k=\lfloor \log_2f(x)\rfloor$. We can find $0=t_0<t_1<t_2<\cdots<t_k$ such that $h(t_{i+1})=h(t_i)/2$ and $h(t_i)\geqslant2$ for any $i=0,1,\dots,k-1$ and $h(t_k)<2$. From the above discussion, we know that
\[d(x,\gamma(t_k))\leqslant\sum_{i=1}^{k}d(\gamma(t_{i-1}),\gamma(t_i))\leqslant c_3\cdot k\leqslant c_3\cdot\log_2 f(x)=c_2\cdot\ln f(x)\]
for a constant $c_2>0$. Moreover, since $h(t_k)<2$, by \prettyref{eq:smalldistestimate} we get that
\[d(x,x_{\mathrm{min}})\leqslant d(x,\gamma(t_k))+d(\gamma(t_k),x_{\mathrm{min}})\leqslant c_2\cdot\ln f(x)+c_1\leqslant c_2\cdot\ln (f(x)+1)+c_1\]
when $f(x)\geqslant2$. Also we have
\[d(x,x_{\mathrm{min}})\leqslant c_1\leqslant c_2\cdot\ln (f(x)+1)+c_1\]
when $f(x)<2$. 
\end{proof}

\section{Proof of Main Results}\label{sec:main}

In this section we prove that stable $\alpha_1$-cyclic parabolic $\SO_0(2,3)$-Higgs bundles with non-zero weight satisfy some domination properties. Following Filip's strategy, to establish the $\log$-Anosov property, we need to choose suitable function $f$ satisfying (S1)-(S3) and use the estimates given in \prettyref{sec:estimates}. 

Below we fix a stable $\alpha_1$-cyclic parabolic $\SO_0(2,3)$-Higgs bundle $(\EE,\Phi)$ represented by \prettyref{eq:cyc}
\[
\begin{tikzcd}
	{\mathcal{L}_{-2}} & {\mathcal{L}_{-1}} & {\mathcal{L}_{0}} & {\mathcal{L}_{1}} & {\mathcal{L}_{2}}
	\arrow["\tau^\vee"', from=1-1, to=1-2]
	\arrow["\beta^\vee"', from=1-2, to=1-3]
	\arrow["\beta"', from=1-3, to=1-4]
	\arrow["\tau"', from=1-4, to=1-5]
	\arrow["\gamma"', curve={height=18pt}, from=1-5, to=1-2]
	\arrow["\gamma^\vee"', curve={height=18pt}, from=1-4, to=1-1]
\end{tikzcd}
\]
satisfying assumption A, of non-zero weights and $\operatorname{pardeg}(\LL_1)<\deg(\KK_\oX(D))$.
By \prettyref{lemma:metricsplit}, the corresponding harmonic metric $h$ splits as $\oplus_{i=-2}^2h_i$ and the corresponding real subbundle $\EE_\mathbb{R}=\oplus_{i=0}^2\left(\LL_i\right)_{\mathbb{R}}$ is given in \prettyref{lemma:realstructure}. The Chern connection $\nabla^{h}$ induced by the harmonic metric and the flat connection $\mathrm{D}^h=\nabla^{h}+\Phi+\Phi^{*_{h}}$ are introduced in \prettyref{fact:harmonicmetricG}. When $\gamma=0$, $\mathrm{D}^h$ is the Gauss--Manin connection in \cite{filip2021uniformization}. 

We lift $(\EE,\Phi)$ to the universal cover on $\widetilde{X}\cong\mathbb{H}^2$ with respect to the flat connection $\mathrm{D}^h$, i.e. $\mathrm{D}^h$ is lifted to the trivial connection $\dd$ on the trivial vector bundle. With a slight abuse of notation, in this section we use $\EE,\Phi,\LL_i,\tau,h,h_{i},(\LL_i)_{\mathbb{R}}$ to denote their lift. 

Now we fix a real vector $v\in(\EE_{\mathbb{R}})_{\widetilde{x_0}}=\left(\bigoplus_{i=0}^2(\LL_i)_{\mathbb{R}}\right)_{\widetilde{x_0}}$ over the basepoint $\widetilde{x_0}\in\widetilde{X}$ satisfying that
\[(h_\UU\oplus(-h_\VV))(v,v)=1.\] It can be extended to a global section $v\colon\mathbb{H}^2\to\EE_{\mathbb{R}}$ with respect to $\mathrm{D}^h$. And it splits into $\sum_{i=-2}^2v_i$, where $v_i\colon\mathbb{H}^2\to\LL_i$ are global smooth sections of $\LL_i$ and $v_{-i}=\overline{v_i}$.  Note that $h_\UU\oplus(-h_\VV)$ is flat along $\mathrm{D}^h$, hence
\[2\|v_1\|_h^2-(2\|v_2\|_h^2+\|v_0\|_h^2)\equiv1.\]
This implies that
$\|v_1\|_h\gtrsim1$ and $\|v_1\|_h\gtrsim\|v_2\|_h$.

Let $f_v:=\|v_2\|_h^2$, we will show that $f_v$ satisfies conditions (S1)-(S3). When $\gamma=0$, this has been proven in \cite[Section 2.2]{filip2021uniformization}.

\subsection{Establish (S1)-(S3) for \texorpdfstring{$f_v$}{fv}}\label{sec:fvS}

It is trivial that $f_v\geqslant0$. Below we first establish conditions (S2) and (S3) for $f_v$ and then prove $f_v$ is a Morse function.

By projecting the equation $\mathrm{D}^h(v)=0$ onto $\LL_2$, we obtain that 
\begin{equation}\label{eq:firstderi}
    \nabla^{h}(v_2)=-\tau(v_1)-\gamma^{*_{h}}(v_{-1}).
\end{equation}

\begin{lemma}
    $f_v$ satisfies conditions (S2)(S3), i.e.
    \[\|\dd f_v\|_{g_{\mathrm{hyp}}^\vee}\gtrsim f_v,\|\dd f_v\|_{g_{\mathrm{hyp}}^\vee}\gtrsim f_v^{1/2}.\]
\end{lemma}

\begin{proof}
    \[\begin{aligned}
        \|\dd f_v\|_{g_{\mathrm{hyp}}^\vee}
        =&\|\dd h(v_2,v_2)\|_{g_{\mathrm{hyp}}^\vee}\\
        =&\|h(\nabla^{h}(v_2),v_2)+h(v_2,\nabla^{h}(v_2))\|_{g_{\mathrm{hyp}}^\vee}\\
        =&\|h(\tau(v_1)+\gamma^{*_h}(v_{-1}),v_2)+h(v_2,\tau(v_1)+\gamma^{*_h}(v_{-1}))\|_{g_{\mathrm{hyp}}^\vee}\quad(\mbox{by \prettyref{eq:firstderi}})\\
        =&\sqrt{2}\cdot\|h(\gamma^{*_h}(v_{-1}),v_2)+h(v_2,\tau(v_1))\|_{g_{\mathrm{hyp}}^\vee}\quad(\mbox{since }\dd z\perp\dd\bar z\mbox{ and }\|\dd z\|=\|\dd\bar z\|)\\
        \geqslant&\sqrt{2}\cdot\left|\|h(v_2,\tau(v_1))\|_{g_{\mathrm{hyp}}^\vee}-\|h(v_{-1},\gamma(v_2))\|_{g_{\mathrm{hyp}}^\vee}\right|\\
        =&\sqrt{2}\cdot\left(\|\tau\|-\|\gamma\|\right)\cdot\|v_1\|_h\cdot\|v_2\|_h\\
        \gtrsim&\|v_1\|_h\cdot\|v_2\|_h\quad(\mbox{by \prettyref{prop:boundedness}})\\
        \gtrsim&\begin{cases}
            \|v_2\|_h^2=f_v&\mbox{since }\|v_1\|_h\gtrsim\|v_2\|_h,\\
            \|v_2\|_h=f_v^{1/2}&\mbox{since }\|v_1\|_h\gtrsim1.
        \end{cases}
    \end{aligned}\]
\end{proof}

\begin{lemma}
    $f_v$ is a Morse function. Furthermore, $f_v$ satisfies the condition (S1).
\end{lemma}

\begin{proof}
    Since $\|\dd f_v\|_{g_{\mathrm{hyp}}^\vee}\gtrsim f_v$, we obtain that the critical points of $f_v$ only occur when $v_2(x)=0$. Now we compute the Hessian of $f_v$ at its critical points. Given two real vector fields $T_1,T_2$ around a critical point $x$. Since $v_2(x)=0$, by the compatibility between $\nabla^{h}$ and the Hermitian metric $h$, one can readily check that
    \begin{equation}\label{eq:secondderi}
        T_1T_2(f_v)(x)=\left(h(\nabla_{T_1}^{h}(v_2),\nabla_{T_2}^{h}(v_2))+h(\nabla_{T_2}^{h}(v_2),\nabla_{T_1}^{h}(v_2))\right)(x).
    \end{equation}
    Now we take a local unit frame $e_2$ of $\LL_2$ around $x$. Locally we have $\tau(v_1)(\partial/\partial z)=se_2$ and $\gamma^{*_h}(v_{-1})(\partial/\partial\bar z)=te_2$ for some complex-valued smooth functions $s$ and $t$. By \prettyref{prop:boundedness}, $\|v_1\|_h=\|v_{-1}\|_h$ and $\|\partial/\partial z\|_{g_{\mathrm{hyp}}}=\|\partial/\partial \bar z\|_{g_{\mathrm{hyp}}}$ we obtain that $|s|>|t|$.

    With respect to the natural coordinate basis $\partial/\partial x,\partial/\partial y$, a quick calculation shows that the coordinate Hessian of $f_v$ at $x_0$ can be represented as
    \[\begin{pmatrix}
        2|s+t|^2&\iu\left[\overline{(s+t)}(s-t)-\overline{(s-t)}(s+t)\right]\\\iu\left[\overline{(s+t)}(s-t)-\overline{(s-t)}(s+t)\right]&2|s-t|^2
    \end{pmatrix}.\]
    It suffices to prove that the determinant of above matrix is not $0$. Actually, the determinant is
    \[\begin{aligned}
        &4\left[|s+t|^2\cdot|s-t|^2\right]+\left[\overline{(s+t)}(s-t)-\overline{(s-t)}(s+t)\right]^2\\=&\left[\overline{(s+t)}(s-t)+\overline{(s-t)}(s+t)\right]^2\geqslant0
    \end{aligned}\]
    and the ``='' holds iff \[\begin{aligned}
        &\overline{(s+t)}(s-t)+\overline{(s-t)}(s+t)=0\\
        \iff&\overline{(s-t)}(s+t)=r\iu\mbox{ for some real }r\\
        \iff& s=\dfrac{r\iu+1}{r\iu-1}t\mbox{ for some real }r\\
        \implies&|s|=|t|,
    \end{aligned}\]
    contradiction.
\end{proof}

\begin{remark}
    The last step above to avoid the equality holds has the following Euclidean geometric illustration: The two diagonals of a parallelogram are perpendicular if and only if the parallelogram is a diamond.
\end{remark}

\begin{remark}
    In \cite{filip2021uniformization}, the associated Higgs bundle of the RVHS has a vanishing $\gamma$, so $v_2$ is a holomorphic section. Then the non-degeneration of the critical points of $f_v$ can be easily proven by the holomorphicity. When $\gamma\neq 0$, we must use the estimates \prettyref{prop:boundedness} of Higgs fields to show the non-degeneration.
\end{remark}

Therefore, $f_v$ satisfies conditions (S1)-(S3) and by \prettyref{lemma:exponentialgrowth} we have the following corollary:

\begin{corollary}\label{coro:exp}
    There exist constants $C_1,C_2,\varepsilon > 0$ independent of $v$ such that
\[f_v(x)\geqslant C_1\cdot\exp(\varepsilon\cdot d(x_{\mathrm{min}},x))-C_2,\forall x\in \widetilde{X},\]
where $x_{\mathrm{min}}$ is the unique point such that $v_2(x_{\mathrm{min}})=0$. 
\end{corollary}

\begin{proof}
    It follows from that the constants appear in
    $\|\dd f_v\|_{g_{\mathrm{hyp}}^\vee}\gtrsim f_v$, $\|\dd f_v\|_{g_{\mathrm{hyp}}^\vee}\gtrsim f_v^{1/2}$ are independent of $v$.
\end{proof}

\subsection{Establish the Domination Property}\label{sec:Anosov}

Below we will prove our main theorem \prettyref{thm:main}. Recall the notation of Cartan projection, almost-dominated representation, non-Abelian Hodge correspondence and $\{\alpha_1\}$-cyclic parabolic $\SO_0(2,3)$-Higgs bundle defined in \prettyref{sec:pre} and \prettyref{sec:Higgsestimate}.

\begin{theorem}\label{thm:soAnosov}
    For any stable $\alpha_1$-cyclic parabolic $\SO_0(2,3)$-Higgs bundle $(\EE,\Phi):=$ 
    % https://q.uiver.app/#q=WzAsNSxbMCwwLCJcXG1hdGhjYWx7TH1fey0yfSJdLFsxLDAsIlxcbWF0aGNhbHtMfV97LTF9Il0sWzIsMCwiXFxtYXRoY2Fse0x9X3swfSJdLFszLDAsIlxcbWF0aGNhbHtMfV97MX0iXSxbNCwwLCJcXG1hdGhjYWx7TH1fezJ9Il0sWzAsMSwiXFxhbHBoYSIsMl0sWzEsMiwiXFxiZXRhIiwyXSxbMiwzLCJcXGJldGEiLDJdLFszLDQsIlxcYWxwaGEiLDJdLFs0LDEsIlxcZ2FtbWEiLDIseyJjdXJ2ZSI6M31dLFszLDAsIlxcZ2FtbWEiLDIseyJjdXJ2ZSI6M31dXQ==
\[\begin{tikzcd}
	{\mathcal{L}_{-2}} & {\mathcal{L}_{-1}} & {\mathcal{L}_{0}} & {\mathcal{L}_{1}} & {\mathcal{L}_{2}}
	\arrow["\tau^\vee"', from=1-1, to=1-2]
	\arrow["\beta^\vee"', from=1-2, to=1-3]
	\arrow["\beta"', from=1-3, to=1-4]
	\arrow["\tau"', from=1-4, to=1-5]
	\arrow["\gamma"', curve={height=18pt}, from=1-5, to=1-2]
	\arrow["\gamma^\vee"', curve={height=18pt}, from=1-4, to=1-1]
\end{tikzcd}\]
satisfying assumption A, of nonzero weights and $\operatorname{pardeg}(\LL_1)<\deg(\KK_\oX(D))$, then its corresponding representation
    $\rho:=\mathsf{NAH}((\EE,\Phi))$ satisfies that \[\alpha_2(\mu(\rho(\sigma)))\geqslant C_3\cdot d(\widetilde{x_0},\widetilde{x_0}\cdot\sigma)-C_4,\]
    where $\mu$ denotes the Cartan projection and $d$ denotes the hyperbolic distance on $\oX$, for some constant $C_3,C_4>0$ which are independent of our choice of $\sigma$, i.e. $\rho$ is $\{\alpha_{2}\}$-almost dominated.
\end{theorem}

\begin{proof}
     Let $V\cong\mathbb{R}^{5}$ be the fiber of $\EE_{\mathbb{R}}$ at $\widetilde{x_0}$. We can choose a basis $\{e_1,e_2,f_2,f_1,f_3\}$ of $\mathbb{R}^{5}$ which is an orthonormal basis of the indefinite billinear form $(h_\UU\oplus(-h_\VV))|_{\EE_{\mathbb{R}}}$ whose signature is $(2,3)$ given by the harmonic metric such that the standard representation $\mathfrak{a}\to\mathfrak{gl}(V)$ is given by
     \begin{equation}\label{eq:matrix}
         A:=\begin{pmatrix}
         0&0&0&a_1&0\\
         0&0&a_2&0&0\\
         0&a_2&0&0&0\\
         a_1&0&0&0&0\\\
         0&0&0&0&0
     \end{pmatrix}\mapsto\begin{pmatrix}
         0&0&0&a_1&0\\
         0&0&a_2&0&0\\
         0&a_2&0&0&0\\
         a_1&0&0&0&0\\\
         0&0&0&0&0
     \end{pmatrix}, 
     \end{equation}
     c.f.
     \prettyref{example:sopq}.
     Hence this representation
     has the following weight space decomposition:
    \[V_{\theta_i}=\mathbb{R}\cdot(e_i+f_i),V_{-\theta_i}=\mathbb{R}\cdot(e_i-f_i),V_0=\mathbb{R}\cdot f_3,\]
    where $\theta_i=[A\mapsto a_i]\in\mathfrak{a}^\vee$ is the linear function defined in  \prettyref{example:sopq}.
For any $\mu\in\overline{\mathfrak{a}^+}$ with $\theta_i(\mu)=:\mu_i$ and $\mu_i\geqslant0$ where $i=1,2$, we obtain that
    \[2\exp(\mu)\cdot e_i=\exp(\mu)\cdot(e_i+f_i)+\exp(\mu)\cdot(e_i-f_i)=\exp(\mu_i)\cdot(e_i+f_i)+\exp(-\mu_i)\cdot(e_i-f_i).\]

    Now we take an arbitrary $\sigma\in\pi_1(X)$ and consider the KAK decomposition of $\rho(\sigma)$, i.e.
    $\rho(\sigma)=k_-\exp(\mu)k_+$, where $k_-,k_+\in K$ and $\mu=\mu(\rho(\sigma))\in\overline{\mathfrak{a}^+}$. We have that
    \[\begin{aligned}
        \|v\|_{h}( \widetilde{x_0}\cdot \sigma)=&\|\rho(\sigma)^{-1}\cdot v\|_{h}(\widetilde{x_0})\\=&\|\exp(-\mu)k_-^{-1}\cdot v\|_{h}(\widetilde{x_0})\quad(k_+^{-1}\mbox{ preserves the harmonic metric})\\=&\|\operatorname{Ad}((k^{op})^{-1})\exp(\mu)k_-^{-1}\cdot v\|_{h}(\widetilde{x_0})\quad(\mbox{opposition involution})\\
        =&\|\exp(\mu)k^\prime\cdot v\|_{h}(\widetilde{x_0})\quad(k^\prime=k^{op}k_-^{-1}\in K).
    \end{aligned}\]
    Now since $K$ preserves $\bigoplus_{i=1}^2\mathbb{R}\cdot {e_i}$, we take $v^\prime=(k')^{-1}\cdot e_2\in\bigoplus_{i=1}^2\mathbb{R}\cdot {e_i}$ (dependent on the choice of $\sigma$). Recall that it extends to a global flat section $v'\colon\mathbb{H}^2\to\EE$ and splits as $\sum_{i=-2}^2(v')_i$ with respect to the decomposition $\EE=\bigoplus_{i=-2}^2\LL_i$. By our choice of the basis we also know that $(v^\prime)_2(\widetilde{x_0})=0$, i.e. $\widetilde{x_0}$ is the unique minima of $f_{v'}$ (c.f. \prettyref{thm:index}). We obtain that
    \[\begin{aligned}
        &\|v^\prime\|_{h}^2(\widetilde{x_0}\cdot\sigma)\\
        =&\|\exp(\mu)\cdot e_2\|_{h}^2(\widetilde{x_0})\\
        =&\dfrac{1}{4}\exp(2\mu_2)\cdot\|e_2+f_2\|_{h}^2(\widetilde{x_0})+\dfrac{1}{4}\exp(-2\mu_2)\cdot\|e_2-f_2\|_{h}^2(\widetilde{x_0})\\
        \leqslant&\dfrac{1}{2}\exp(2\mu_2)+\dfrac{1}{2}\quad(\mbox{since }\|e_2+f_2\|_{h}^2=\|e_2-f_2\|_{h}^2=2,\ \mu_2\geqslant0).
    \end{aligned}\]

    On the other hand, 
    \[\begin{aligned}
        &\|v^\prime\|_{h}^2(\widetilde{x_0}\cdot\sigma)\\
        \geqslant&\|(v^\prime)_2\|_{h}^2(\widetilde{x_0}\cdot\sigma)\\
        =&f_{v'}(\widetilde{x_0}\cdot\sigma)\\
        \geqslant& C_1\cdot\exp(d(\widetilde{x_0},\widetilde{x_0}\cdot\sigma))-C_2\quad(\mbox{by \prettyref{coro:exp}}).
    \end{aligned}\]
    Note that here $C_1,C_2$ are independent of our choice of $\sigma$. Therefore we have that 
    \[\alpha_2(\mu(\rho(\sigma)))=\mu_2\geqslant C_3\cdot d(\widetilde{x_0},\widetilde{x_0}\cdot\sigma)-C_4\]
    for some constant $C_3,C_4>0$ which are independent of our choice of $\sigma$, which implies that $\rho$ is $\{\alpha_2\}$-almost dominated.
\end{proof}

Now the main result \prettyref{thm:main} follows from \prettyref{thm:soAnosov}, \prettyref{prop:stable} and \prettyref{fact:equi}.

\begin{remark}\label{rem:higher}
    We should point out that we do not use the condition that $\operatorname{rank}(\LL_0)=1$ except in \prettyref{prop:stable}! One can freely change $\LL_0$ into a parabolic orthogonal vector bundle of rank $n$ whose underlying bundle has trivial determinant in all other results. In particular, if the resulting parabolic $\SO_0(2,n+2)$-Higgs bundle
    \[\begin{tikzcd}
	{\mathcal{L}_{-2}} & {\mathcal{L}_{-1}} & {\mathcal{L}_{0}} & {\mathcal{L}_{1}} & {\mathcal{L}_{2}}
	\arrow["\tau^\vee"', from=1-1, to=1-2]
	\arrow["\beta^\vee"', from=1-2, to=1-3]
	\arrow["\beta"', from=1-3, to=1-4]
	\arrow["\tau"', from=1-4, to=1-5]
	\arrow["\gamma"', curve={height=18pt}, from=1-5, to=1-2]
	\arrow["\gamma^\vee"', curve={height=18pt}, from=1-4, to=1-1]
\end{tikzcd}\]
is still stable, with nonzero weights and $\operatorname{pardeg}(\LL_1)<\deg(\KK_\oX(D))$, then we also have 
\begin{itemize}
    \item[(1)] $\|\tau\|$ is bounded by positive constants since the graded residue is still lower-triangular; (\prettyref{lemma:completeness})

    \item[(2)] $\|\tau\|-\|\gamma\|>C$ for a positive constant $C$ because the Hitchin's self-dual equation on $\LL_{-2}$ has the same form. (\prettyref{prop:boundedness})
\end{itemize}
These imply that the corresponding representation $\pi_1(X)\to\SO_0(2,n+2)$ is $\{\alpha_2\}$-almost dominated representation by applying the process in \prettyref{sec:main} again. 
\end{remark}

%~~~~~~~~~ Appendix

%\appendix

\paragraph{Data availability} Data sharing does not apply to this article as no datasets were generated or analyzed during the current research.

\paragraph{Conflict of Interest} The author stated that there is no Conflict of interest.

%~~~~~~~~~ Bibliography

%\nocite{*}
\bibliographystyle{plain}
\bibliography{bib}

\begin{thebibliography}{10}

\bibitem{alias2016maximum}
Luis~J Al{\'\i}as, Paolo Mastrolia, and Marco Rigoli.
\newblock {\em Maximum principles and geometric applications}, volume 700.
\newblock Springer, 2016.

\bibitem{aparicio2019so}
Marta Aparicio-Arroyo, Steven Bradlow, Brian Collier, Oscar Garc{\'\i}a-Prada, Peter~B Gothen, and Andr{\'e} Oliveira.
\newblock $\mathrm{SO}(p,q)$-{H}iggs bundles and higher {T}eichm{\"u}ller components.
\newblock {\em Inventiones mathematicae}, 218(1):197--299, 2019.

\bibitem{biquard2020parabolic}
Olivier Biquard, Oscar Garc{\'\i}a-Prada, and Ignasi~Mundet i~Riera.
\newblock Parabolic {H}iggs bundles and representations of the fundamental group of a punctured surface into a real group.
\newblock {\em Advances in Mathematics}, 372:107305, 2020.

\bibitem{biquard2017higgs}
Olivier Biquard, Oscar Garc{\'\i}a-Prada, and Roberto Rubio.
\newblock Higgs bundles, the {T}oledo invariant and the {C}ayley correspondence.
\newblock {\em Journal of Topology}, 10(3):795--826, 2017.

\bibitem{bisgard2015mountain}
James Bisgard.
\newblock Mountain passes and saddle points.
\newblock {\em Siam Review}, 57(2):275--292, 2015.

\bibitem{bochi2019anosov}
Jairo Bochi, Rafael Potrie, and Andr{\'e}s Sambarino.
\newblock Anosov representations and dominated splittings.
\newblock {\em Journal of the European Mathematical Society}, 21(11):3343--3414, 2019.

\bibitem{burger2010higher}
Marc Burger, Alessandra Iozzi, and Anna Wienhard.
\newblock Higher {T}eichm\"uller spaces: from $\mathrm{SL}(2,\mathbb{R})$ to other {L}ie groups.
\newblock {\em arXiv preprint arXiv:1004.2894}, 2010.

\bibitem{burger2010surface}
Marc Burger, Alessandra Iozzi, and Anna Wienhard.
\newblock Surface group representations with maximal {T}oledo invariant.
\newblock {\em Annals of Mathematics}, 172:517--566, 2010.

\bibitem{cheng1975differential}
Shiu~Yuen Cheng and Shing-Tung Yau.
\newblock Differential equations on {R}iemannian manifolds and their geometric applications.
\newblock {\em Communications on Pure and Applied Mathematics}, 28(3):333--354, 1975.

\bibitem{collier2016maximal}
Brian Collier.
\newblock Maximal $\mathrm{Sp}(4,\mathbb{R})$ surface group representations, minimal immersions and cyclic surfaces.
\newblock {\em Geometriae Dedicata}, 180:241--285, 2016.

\bibitem{collier2017asymptotics}
Brian Collier and Qiongling Li.
\newblock Asymptotics of {H}iggs bundles in the {H}itchin component.
\newblock {\em Advances in Mathematics}, 307:488--558, 2017.

\bibitem{collier2019geometry}
Brian Collier, Nicolas Tholozan, and J{\'e}r{\'e}my Toulisse.
\newblock The geometry of maximal representations of surface groups into $\mathrm{SO}_0(2,n)$.
\newblock {\em Duke Mathematical Journal}, 168(15):2873--2949, 2019.

\bibitem{collier2023holomorphic}
Brian Collier and J{\'e}r{\'e}my Toulisse.
\newblock Holomorphic curves in the 6-pseudosphere and cyclic surfaces.
\newblock {\em arXiv preprint arXiv:2302.11516}, 2023.

\bibitem{davalo2024maximal}
Colin Davalo.
\newblock Maximal and {B}orel {A}nosov representations into $\mathrm{Sp}(4,\mathbb{R})$.
\newblock {\em Advances in Mathematics}, 442:109578, 2024.

\bibitem{feng2023compact}
Yu~Feng and Junming Zhang.
\newblock Compact {R}elative $\mathrm{SO}_0(2,q)$-{C}haracter varieties of {P}unctured {S}pheres.
\newblock {\em arXiv preprint arXiv:2309.15553}, 2023.

\bibitem{filip2021uniformization}
Simion Filip.
\newblock Uniformization of some weight 3 variations of {H}odge structure, {A}nosov representations, and {L}yapunov exponents.
\newblock {\em arXiv preprint arXiv:2110.07533}, 2021.

\bibitem{garcia2009hitchin}
Oscar Garcia-Prada, Peter~B Gothen, et~al.
\newblock The {H}itchin-{K}obayashi correspondence, {H}iggs pairs and surface group representations.
\newblock {\em arXiv preprint arXiv:0909.4487}, 2009.

\bibitem{hitchin1987self}
Nigel~J Hitchin.
\newblock The self-duality equations on a {R}iemann surface.
\newblock {\em Proceedings of the London Mathematical Society}, 3(1):59--126, 1987.

\bibitem{kapovich2018morse}
Michael Kapovich, Bernhard Leeb, and Joan Porti.
\newblock A {M}orse lemma for quasigeodesics in symmetric spaces and {E}uclidean buildings.
\newblock {\em Geometry \& Topology}, 22(7):3827--3923, 2018.

\bibitem{kim2018analytic}
Semin Kim and Graeme Wilkin.
\newblock Analytic convergence of harmonic metrics for parabolic {H}iggs bundles.
\newblock {\em Journal of Geometry and Physics}, 127:55--67, 2018.

\bibitem{knapp1996lie}
Anthony~W Knapp.
\newblock {\em Lie groups beyond an introduction}, volume 140.
\newblock Springer, 1996.

\bibitem{labourie2006anosov}
Fran{\c{c}}ois Labourie.
\newblock Anosov flows, surface groups and curves in projective space.
\newblock {\em Inventiones mathematicae}, 165(1):51--114, 2006.

\bibitem{labourie2017cyclic}
Fran{\c{c}}ois Labourie.
\newblock Cyclic surfaces and {H}itchin components in rank 2.
\newblock {\em Annals of Mathematics}, 185(1):1--58, 2017.

\bibitem{simpson1990harmonic}
Carlos~T Simpson.
\newblock Harmonic bundles on noncompact curves.
\newblock {\em Journal of the American Mathematical Society}, 3(3):713--770, 1990.

\bibitem{yokogawa1993compactification}
K{\^o}ji Yokogawa.
\newblock Compactification of moduli of parabolic sheaves and moduli of parabolic {H}iggs sheaves.
\newblock {\em Journal of Mathematics of Kyoto University}, 33(2):451--504, 1993.

\bibitem{zhu2021relatively}
Feng Zhu.
\newblock Relatively dominated representations.
\newblock {\em Annales de l'Institut Fourier}, 71(5):2169--2235, 2021.

\end{thebibliography}

\end{document}